\documentclass{amsart} 

\usepackage{amsthm, amssymb, amsmath, amscd}
\usepackage{eurosym}
\usepackage{amsfonts}
\usepackage{amsmath}
\usepackage{setspace}
\usepackage{bbm}
\usepackage{graphicx}

\usepackage[all,cmtip]{xy}

\usepackage{indentfirst}

\newtheorem{theorem}{Theorem}[section]

\newtheorem{cor}[theorem]{Corollary}

\newtheorem{lemma}[theorem]{Lemma}

\newtheorem{prop}[theorem]{Proposition}

\theoremstyle{definition}
\newtheorem{define}[theorem]{Definition}

\newtheorem{remark}[theorem]{Remark}
\newtheorem{notation}[theorem]{Notation}

\newcommand{\id}{\mathrm{id}}
\newcommand{\ch}{\mathrm{ch}}
\newcommand{\cs}{\mathrm{cs}}
\newcommand{\ind}{\mathrm{ind}}

\newcommand{\field}[1]{\mathbb{#1}}

\newcommand{\C}{\field{C}}
\newcommand{\R}{\field{R}}
\newcommand{\Z}{\field{Z}}
\newcommand{\rz}{\field{R}/\field{Z}}

\begin{document}
\title[Realizing the analytic surgery group geometrically: Part II]{Realizing the analytic surgery group of Higson and Roe geometrically\\
Part II: Relative \Large$\eta$\normalsize-invariants}
\author{Robin J. Deeley, Magnus Goffeng}
\date{\today}

\begin{abstract}
We apply the geometric analog of the analytic surgery group of Higson and Roe to the relative $\eta$-invariant.  In particular, by solving a Baum-Douglas type index problem, we give a ``geometric" proof of a result of Keswani regarding the homotopy invariance of relative $\eta$-invariants. The starting point for this work is our previous constructions in \emph{``Realizing the analytic surgery group of Higson and Roe geometrically, Part I: The geometric model"}.
\end{abstract}

\maketitle

\section*{Introduction}
This is the second in a series of three papers on the construction and applications of a Baum-Douglas type model for the analytic surgery exact sequence of Higson and Roe \cite{HRSur1, HRSur2, HRSur3}. In the first paper \cite{paperI}, we defined the geometric cycles, proved that the associated group fits into the correct exact sequence, and constructed maps to various $K$-theory groups; these maps are the geometric counterpart and generalization of the map from the analytic surgery group to the real numbers constructed in \cite{HReta}.

The main goal of this paper is to relate the geometric group and the aforementioned maps to the relative $\eta$-invariant; the reader should notice a strong analogy between this paper and the constructions in \cite{HReta}. We also lay the foundation for constructions that will appear in the next paper in the series. The main topic of the next paper is the construction of an explicit isomorphism between the geometric group and Higson and Roe's analytic group; this map is defined from geometric to analytic cycles.

Furthermore, in this paper, we begin the process of applying ``Baum's approach to index theory" to vanishing results concerning the relative $\eta$-invariant. The vanishing results we are interested in are well studied. For example, Keswani \cite[Theorem 1.2]{Kes} has proved the following beautiful result (see \cite{Kes} for notation):

\begin{theorem}
\label{keswanithm}
Let M be a closed, smooth, oriented, odd-dimensional, Riemannian manifold. Suppose that $\Gamma$ is a torsion-free, discrete group and the Baum-Connes assembly map is an isomorphism for the full group $C^*$-algebra of $\Gamma$. Let $\sigma$ be a finite-dimensional unitary representation of $\pi_1(M)$ that factors through $\Gamma$. Then the relative $\eta$-invariant $\rho_{\sigma}(M)$ is an oriented homotopy invariant of M.
\end{theorem}
\noindent
The work of Keswani was preceded by work of Mathai \cite{Mat}, Neumann \cite{Neu} and Weinberger \cite{Wei}. Following Keswani, similar results have been obtained by Benameur-Piazza \cite{BP}, Benameur-Roy \cite{BR, BRvonN}, Higson-Roe \cite{HReta}, Piazza-Schick \cite{PS}, Wahl \cite{WahlSurSet}, among others.

Although our main goal is a proof of Keswani's result (i.e., Theorem \ref{keswanithm}), the methods presented provide a template for proving similar results for variants of the relative $\eta$-invariant (e.g, Cheeger-Gromov's $\ell^2$-relative $\eta$-invariants); the reader can find a ``Higson and Roe" type proof of the main result of \cite{KesvN} in the recent paper of Benameur-Roy \cite{BRvonN}. We discuss this idea in a bit more detail in Remark \ref{finalremark513}. Furthermore, in the next paper in the series higher invariants will be considered in the context of our geometric cycles (see the work of Piazza-Schick \cite{PS} for more on the analytic context). In fact, a detailed study of the geometrically defined invariants constructed in \cite{paperI} is planned in future work. 

We begin with a short review of the ``Baum approach to index theory" in the context of geometric and analytic $K$-homology. The reader can find more details on this method in \cite{BD}; the introduction to \cite{BE} also discusses this approach to index theory in detail. Let $X$ be a finite CW-complex, $K^{ana}_*(X)$ be its analytic $K$-homology group, and $K^{geo}_*(X)$ be its geometric $K$-homology group; the reader can find the definitions of the cycles which determine these two groups in (for example) \cite{BD, BHS}. The starting point for this approach to index theory is the following commutative diagram:
\begin{equation}
\label{commdiagindprob}
\xymatrix{
 K^{{\rm geo}}_0(X) \ar[rdd]_{{\rm ind}_{{\rm top}}}  
 \ar[rr]^{\mu} 
 & & K^{{\rm ana}}_0(X)  \ar[ldd]^{{\rm ind}_{{\rm ana}}}  
 \\ \\
& \field{Z} &
}
\end{equation}
where the maps ${\rm ind}_{{\rm top}}$ and ${\rm ind}_{{\rm ana}}$ are respectively the topological and analytic index maps, and $\mu$ is the natural isomorphism (in even degree) between geometric and analytic $K$-homology (again the reader can find further details in \cite{BD, BHS}). We emphasize the direction of the map $\mu$ at the level of cycles: it takes geometric cycles to analytic cycles.

The ``index problem", stated in \cite[page $154$]{BD}, is defined as follows. As input, one takes an explicit analytic cycle yielding a class in $K^{ana}_*(X)$. The ``problem" is to find an explicit geometric cycle (yielding a class in $K^{geo}_*(X)$) with the property that it maps to the class of the analytic cycle (i.e., the input) under the map $\mu$. From such a geometric cycle (and the commutative diagram \eqref{commdiagindprob}), we obtain a topological formula for the index of the input; it is given by the image of the geometric cycle under the map ${\rm ind}_{{\rm top}}$. This method has been used in the context of the Atiyah-Singer index theorem in \cite{BD} and (more recently) in the case of an index theorem for hypoelliptic operators in \cite{BE}. To apply this method, it is required that one has a map from the geometric group to another abelian group; in most cases, the image group is an analytic group, but in this paper we will use a ``geometrically" defined group, which was introduced by Higson and Roe in \cite{HReta}. The more standard Baum approach to index theory will be applied to our situation in the next paper in the series.\\

The geometric model for analytic surgery was constructed in the first paper in this series \cite{paperI}; we recall it briefly and refer the reader to \cite{paperI} for further details. The input data for the geometric model of \cite{paperI} is a discrete group $\Gamma$, a Banach algebra completion $\mathcal{A}(\Gamma)$ of $\C[\Gamma]$, and the associated assembly map; we denote the group produced by our construction by $\mathcal{S}_*^{geo}(\Gamma,\mathcal{A})$. The prototypical cycle for $\mathcal{S}^{geo}_*(\Gamma,\mathcal{A})$ consists of a cycle for $K_{*-1}^{geo}(B\Gamma)$ whose image under the assembly map vanishes  in $K_{*-1}(\mathcal{A}(\Gamma))$ and a bordism implementing this vanishing. To be more precise, a ``nice" class of cycles are those of the form $(W,(\mathcal{E}_{\mathcal{A}(\Gamma)},E_\C,\alpha),f)$ where 
\begin{enumerate}
\item $W$ is a compact spin$^c$-manifold with boundary;
\item $f:\partial W\to B\Gamma$ is a continuous map;
\item $\mathcal{E}_{\mathcal{A}(\pi)}\to W$ is a locally trivial bundle of finitely generated projective $\mathcal{A}(\Gamma)$-modules;
\item $E_\C\to \partial W$ is a (Hermitian) vector bundle;
\item $\alpha$ implements an $\mathcal{A}(\Gamma)$-linear isomorphism of bundles 
$$\mathcal{E}_{\mathcal{A}(\Gamma)}|_{\partial W}\cong E_\C\otimes f^*\mathcal{L}_\mathcal{A},$$ 
where $\mathcal{L}_\mathcal{A}:=E\Gamma\times_\Gamma\mathcal{A}(\Gamma)\to B\Gamma$ is the Mischenko bundle. 
\end{enumerate}
Cycles of the form above are referred to as easy cycles. The general cycles for $\mathcal{S}_*^{geo}(\Gamma,\mathcal{A})$ are of the form $(W,\xi,f)$ where $\xi$ is a cocycle for a relative group for assembly on the level of $K$-theory. The details of the relative $K$-theory groups for assembly can be found in \cite[Section $1.3$]{paperI}. Let us briefly recall that a relative $K$-theory cocycle for assembly on $(W,\partial W,f)$ is a quintuple $\xi=(\mathcal{E}_{\mathcal{A}(\Gamma)},\mathcal{E}_{\mathcal{A}(\Gamma)}',E_\C,E_\C',\alpha)$ where $\mathcal{E}_{\mathcal{A}(\Gamma)}, \mathcal{E}_{\mathcal{A}(\Gamma)}'\to W$ and $E_\C, E_\C'\to \partial W$ are, as for the easy cycles, bundles for the Banach algebra indicated in the subscript and $\alpha$ is an $\mathcal{A}(\Gamma)$-linear isomorphism of bundles
\begin{equation}
\label{relktheocycle}
\alpha:\mathcal{E}_{\mathcal{A}(\Gamma)}|_{\partial W}\oplus (E_\C' \otimes_\C f^*\mathcal{L})\xrightarrow{\sim} \mathcal{E}_{\mathcal{A}(\Gamma)}'|_{\partial W}  \oplus(E_\field{C}\otimes_\field{C} f^*\mathcal{L}).
\end{equation}

The length of this paper is (more or less) explained by the computational complications caused by having to consider cycles of this more general form. It is therefore an interesting open question whether the geometric group can be defined using only the easy cycles discussed above. \\

The content of the paper is as follows. After reviewing some relevant aspects of the Atiyah-Patodi-Singer index theorem in Section \ref{relativeetasection}, we recall (in Subsection \ref{sectionorientedmodel}) the definition of an ``oriented" version of the geometric $K$-homology from \cite{guentnerkhom,HReta,Kescont}. From two finite-dimensional unitary representations $\sigma_1$ and $\sigma_2$ of $\Gamma$ of the same rank, a group $\mathcal{S}^{h}_1(\sigma_1,\sigma_2)$ consisting of equivalence classes of ``oriented" cycles relative to $(\sigma_1,\sigma_2)$ was constructed in \cite{HReta}. It provides a hybrid model (cycles contain both geometric and analytic data) that encode the data required to define relative $\eta$-invariants (with respect to the representations, $\sigma_1$ and $\sigma_2$). A cycle in this model consists of a quintuple $(M,S_{\C\ell},f,D,n)$ where $(M,S_{\C\ell},f)$ is a cycle for the \emph{oriented model} $K_1^h(B\Gamma)$ of $K$-homology, as such $M$ is oriented and $S_{\C\ell}\to M$ is a Clifford bundle, $D$ is a Dirac operator on $(M,S_{\C\ell},f)$ and $n$ is an integer; the dependence of this group on the choice of representations $\sigma_1$ and $\sigma_2$ only appears in the relations used to defined the group (see \cite[Remark 8.8]{HReta}). We review the definition of this group in Section \ref{higsonroesgroup} and extend the class of cycles to also encompass a spectral section; these cycles are better suited for the study of stable relative $\eta$-invariants. 

In Section \ref{mappingtohrrelgroup}, we construct a map $\Phi_\mathcal{A}$ from the geometric group $\mathcal{S}^{geo}_0(\Gamma,\mathcal{A})$ to $\mathcal{S}^{h}_1(\sigma_1, \sigma_2)$ under suitable assumptions on the Banach algebra $\mathcal{A}(\Gamma)$. To explain the idea of this mapping, we restrict to the easy cycles described above. The map, $\Phi_\mathcal{A}$, is defined (at the level of easy cycles) via
\begin{equation}
\label{phioneasy}
(W,(\mathcal{E}_{\mathcal{A}(\Gamma)},E_\C,\alpha),f)\mapsto (\partial W,S_{\partial W}\otimes E_\C,f,D,n)
\end{equation}
where $S_{\partial W}$ denotes the spin$^c$-structure on $\partial W$, $D$ is a choice of Dirac operator on $S_{\partial W}\otimes E$ and $n$ is the index of an Atiyah-Patodi-Singer type operator associated with the bundle data $(\mathcal{E}_{\mathcal{A}(\Gamma)},E_\C,\alpha)$ and the choice of Dirac operator $D$. We also provide an alternative definition $\tilde{\Phi}_{C^*_{{\bf full}}}$ of the map $\Phi_{C^*_{{\bf full}}}$ to $\mathcal{S}^{h}_1(\sigma_1,\sigma_2)$ in Subsection \ref{phiaforcstar} based on higher Atiyah-Patodi-Singer index theory by encoding a choice of spectral section in the data. As such, unlike the geometric model for analytic surgery, the map $\tilde{\Phi}_{C^*_{{\bf full}}}$ is only defined for $C^*$-algebra coefficients. We restrict our attention to the full (maximal) $C^*$-completion $C^*_{{\bf full}}(\Gamma)$ of $\C[\Gamma]$. 

We compare the map $\Phi_{C^*_{{\bf full}}}$, composed with the relative $\eta$-invariant that defines a map $\mathcal{S}^{h}_1(\sigma_1,\sigma_2)\to \R$, to the map to the real number from \cite{paperI} (mentioned above) in Section \ref{relatingtoindaleph}. The map from \cite{paperI} depends on a set of data $\aleph$ constructed from $\sigma_1$ and $\sigma_2$ and gives a map 
$$\ind_\aleph^\R:\mathcal{S}^{geo}_0(\Gamma,C^*_{{\bf full}})\to K_0(N) =\R.$$ 
Here $N$ denotes a II$_1$-factor. The two maps differ by a Chern-Simons term 
$$\cs_{\aleph}:\mathcal{S}^{geo}_0(\Gamma,C^*_{{\bf full}})\to K_0(N)=\R.$$ 
This situation is analogous to that of the various $\rz$-valued pairings considered in \cite{DeeRZ}. The Chern-Simons term compensates for the fact that the map $\ind_\aleph^\R$ depends on a choice of \emph{cocycle} representing the canonical class that $\sigma_1$ and $\sigma_2$ defines in $K^1(B\Gamma;\rz)$, see \cite{APS3,antazzska}. In summary, we have following result (which appears as Theorem \ref{indandrho} in the main body of the paper):

\begin{theorem}
\label{reallythm43}
Assume that $\sigma_1,\sigma_2:\Gamma\to U(k)$ are representations and let $\aleph$ and  $\aleph_0$ be the data chosen in the introduction of Section \ref{relatingtoindaleph} (cf. \cite[Section $5$]{paperI}). The $\aleph$-Chern-Simons invariant of cycles with connection from Definition \ref{csdef} is well defined and induces a map $\cs_{\aleph}:\mathcal{S}^{geo}_0(\Gamma,C^*_{{\bf full}})\to \R$ fitting into the commutative diagram:
\begin{center}
$$\xymatrix{
\mathcal{S}_0^{geo}(\Gamma,C^*_{{\bf full}})\ar[rr]^{\Phi_{C^*_{{\bf full}}}}\ar[ddr]_{(\tau_N)_*\circ \ind_\aleph^\R-\cs_{\aleph}}& &\mathcal{S}_1^h(\sigma_1,\sigma_2)\ar[ddl]^{\rho_{\sigma_1,\sigma_2}}
\\ \\ 
 &\R&  
}.$$
\end{center}
where
\begin{enumerate}
\item $\Phi_{C^*_{\bf full}}$ is the map in Definition \ref{defOfPhi};
\item $\rho_{\sigma_1, \sigma_2}$ is the map in Definition \ref{definitionofrho};
\item $\tau_N$, ${\rm ind}^{\field{R}}_{\aleph}$, and $cs_{\aleph}$ are the maps discussed in Section \ref{relatingtoindaleph}.
\end{enumerate}
\end{theorem}

The map $\ind_\aleph^\R-\cs_{\aleph}$ can be viewed as analogous to the topological index map and $\rho_{\sigma_1,\sigma_2}\circ \Phi_{C^*_{\bf full}}$ as analogous to the analytic index map in the  ``Baum approach to index theory" discussed above. The application of this approach to the rigidity of relative $\eta$-invariants is based on the following theorem describing how the mappings of Theorem \ref{reallythm43} fit together with the geometric version of Higson-Roe's analytic surgery exact sequence (see \cite[Theorem 3.8]{paperI}).

\begin{theorem}
\label{fittingtogether}
Assume that $\sigma_1,\sigma_2:\Gamma\to U(k)$ are representations. The following diagram has exact rows and commutes
\begin{center}
$$\xymatrix{
&K_*^{geo}(pt;C^*_{{\bf full}}(\Gamma))\ar[r]^{r}\ar[d]^{(\sigma_1-\sigma_2)_*}&\mathcal{S}_0^{geo}(\Gamma,C^*_{{\bf full}})\ar[d]^{\Phi_{C^*_{{\bf full}}}}\ar[r]^{\delta}& K_{*-1}^{geo}(B\Gamma)\ar@{=}[d]& \\
0\ar[r]&\Z\ar[r]\ar@{=}[d] &\mathcal{S}_1^h(\sigma_1,\sigma_2)\ar[d]^{\rho_{\sigma_1,\sigma_2}}\ar[r]&K_{*-1}^{geo}(B\Gamma)\ar[d]^{\bar{\rho}_{\sigma_1,\sigma_2}}\ar[r]&0\\ 
0\ar[r]&\Z\ar[r] &\R\ar[r]&\R/\Z\ar[r]&0  
}.$$
\end{center}
Here the mappings $r$ and $\delta$ are described in \cite[Section 3]{paperI}, $\Phi_{C^*_{\bf full}}$ is the map in Definition \ref{defOfPhi}, $\rho_{\sigma_1, \sigma_2}$ is the map in Definition \ref{definitionofrho} and $\bar{\rho}_{\sigma_1, \sigma_2}$ denotes the $\mod \Z$-reduction of $\rho_{\sigma_1, \sigma_2}$.
\end{theorem}

The two bottom rows of the diagram in Theorem \ref{fittingtogether} are described in Proposition \ref{sesfortherelativegroups}. The top two rows are described in Theorem \ref{phiacommutingdiagram}. With Theorem \ref{reallythm43} and Theorem \ref{fittingtogether} in hand, we proceed to apply ``Baum's approach to index theory" to our situation in Section \ref{sectionvanishing}. In particular, our proof of Theorem \ref{keswanithm} is based on solving such an index problem. Given $(M,S_{\C\ell},f,D)$, a cycle for $K_1^h(B\Gamma)$ with a choice of Dirac operator, the \emph{$(\sigma_1,\sigma_2)$-relative index problem} asks for a cycle $(W,\xi,f)$ for $\mathcal{S}_0^{geo}(\Gamma,\mathcal{A})$ such that 
$$\Phi_\mathcal{A}(W,\xi,f)=[M,S_{\C\ell},f,D,0]\quad\mbox{in}\quad \mathcal{S}^h_1(\sigma_1,\sigma_2).$$ 
Whenever the index problem for $(M,S_{\C\ell},f,D)$ admits a solution and $\Gamma$ is a torsion-free group for which $\mathcal{A}(\Gamma)$ has the Baum-Connes property, we have that $\mathcal{S}^{geo}_0(\Gamma,\mathcal{A})=0$; it follows that
\[\rho_{\sigma_1,\sigma_2}(M,S_{\C\ell},f,D)=\rho_{\sigma_1,\sigma_2}\circ \Phi_\mathcal{A}(W,\xi,f)=0.\]
Whenever $(M,S_{\C\ell},f,D)$ admits a solution to the index problem it follows that $\mu(M,S_{\C\ell},f)=0$ in $K_1(\mathcal{A}(\Gamma))$. This is unfortunately not a sufficient condition; the reader can deduce counterexamples from the results of \cite[Section 15]{PS}, see more in Remark \ref{obstructingtheindexproblem}. 

Following ideas in \cite{PS}, we avoid these obstructions for the case of $C^*$-algebra completions of $\C[\Gamma]$ in Subsection \ref{stabelsubsection} through a stable reformulation of the $(\sigma_1,\sigma_2)$-relative index problem using higher Atiyah-Patodi-Singer index theory. The stable $(\sigma_1,\sigma_2)$-relative index problem admits a solution whenever $\mu(M,S_{\C\ell},f)=0$, see Theorem \ref{solvingindexproblem}. It can be concluded that if $(M,S_{\C\ell},f,D)$ vanishes under assembly, $\Gamma$ is torsion-free and has the \emph{full} Baum-Connes property, the stable relative $\eta$-invariant $\rho^s_{\sigma_1,\sigma_2}(M,S_{\C\ell},f,D)=0$. Theorem \ref{keswanithm} follows from techniques in \cite{PS}. \\

\section{Preliminaries on relative $\eta$-invariants}
\label{relativeetasection}

In the seminal work of Atiyah-Patodi-Singer \cite{APS1,APS2,APS3}, a non-local boundary condition on Dirac operators (now known as Atiyah-Patodi-Singer boundary conditions) was studied. We recall this boundary condition below in \eqref{APScondition}. The index formula from \cite{APS1} for a Dirac operator on a compact manifold with boundary equipped with Atiyah-Patodi-Singer boundary conditions contains not only the ordinary local contributions but also a spectral term from the boundary, known as an $\eta$-invariant. It turns out that many interesting secondary invariants can be constructed from $\eta$-invariants. In this section, we briefly recall these invariants and some $C^*$-algebraic techniques introduced in \cite{PS} to study these invariants. We also recall the oriented model for geometric $K$-homology in Subsection \ref{sectionorientedmodel}.

Throughout the paper, $W$ denotes a Riemannian oriented manifold with boundary which often is equipped with a spin$^c$-structure. The bundle $S\to W$ (in this subsection) is assumed to be a Clifford bundle equipped with a Hermitean Clifford connection $\nabla_S$. We sometimes write $S_{\C\ell}$ to emphasize that $S$ is a Clifford bundle. For the purposes of this paper, we always assume that the metric on $W$ and the Clifford connection $\nabla_S$ are of product type near $\partial W$. The associated Dirac operator will be denoted by $D_S$ or $D_{\nabla_S}$.

\subsection{The Atiyah-Patodi-Singer index theorem}
\label{apssubsub}

Assume that the manifold with boundary, $W$, is even-dimensional. Since we have assumed $\nabla_S$ and the metric on $W$ to be of product type near $\partial W$, $D_S$ decomposes near the boundary as follows. We take a collar neighborhood $[0,1]\times \partial W$ of the boundary of $W$ on which all geometric structure are of product type. In particular, we can write $S|_{\partial W}\cong S_N\hat{\otimes}S_\partial$ where $S_N$ is the spin$^c$-structure of the normal bundle to $\partial W$ and $S_\partial$ is a Clifford bundle on $\partial W$. The Clifford bundle $S$ splits into its graded components $S_+\oplus S_-$ and $S_\partial$ can be identified with $S_+|_{\partial W}$. There is an odd self adjoint unitary $\sigma:S|_{\partial W}\to S|_{\partial W}$ acting as Clifford multiplication by $\mathrm{d} t$, where $t$ denotes the coordinate normal to $\partial W$. The Dirac operator $D$ is odd and decomposes into two operators $D^+$ and $D^-$ acting from $S_+\to S_-$ and $S_-\to S_+$ respectively. There is a Dirac operator $D_\partial $ on $S_\partial$ such that near $\partial W$:
\begin{equation}
\label{diracofproducttupe}
D^+=\sigma|_{S_+}\left(\frac{\partial}{\partial t}+D_\partial\right).
\end{equation}
The Dirac operator $D_\partial$ is a self-adjoint elliptic differential operator. Hence, the projection onto the non-negative spectrum of $D_\partial$, written as $P_\partial:=\chi_{[0,\infty)}(D_{\partial})$, is a well defined classical pseudo-differential operator of order $0$ on $\partial W$ acting on the vector bundle $S_\partial$. The Atiyah-Patodi-Singer condition on an element of the first Sobolev space $f\in H^1(W,S_+)$ is given by 
\begin{equation}
\label{APScondition}
P_\partial[f|_{\partial W}]=0.
\end{equation}
The operator $D_S^+$ equipped with the boundary condition \eqref{APScondition} will be denoted by $D_{APS}^+$. By \cite[Theorem $3.10$]{APS1}, the operator $D_{APS}^+$ is Fredholm. To describe its index, we need the $\eta$-function of $D_\partial$; it is defined (for $\mathrm{Re}(s)>\dim(W)-1$) by the expression
\[\eta_{D_\partial}(s):=\sum_{\lambda\in \mathrm{Spec}(D_\partial)\setminus\{0\}}\mathrm{sign}(\lambda)|\lambda|^{-s}.\]
This function extends meromorphically to the complex plane; by \cite[Theorem $3.10$]{APS1}, it is regular at $s=0$. Furthermore, the Atiyah-Patodi-Singer index formula computes the index of $D_{APS}^+$:
\begin{equation}
\label{APSformula}
\ind(D_{APS}^+)=\int_W\alpha_D-\frac{1}{2}\left(\eta_{D_\partial}(0)+\dim\ker D_\partial\right),
\end{equation}
where $\alpha_D$ is a density coming from a local expression. For instance, when $S$ is of the form $S_W\otimes E$, for a spin$^c$-structure $S_W\to W$ with Clifford connection $\nabla_W$ and a hermitean vector bundle $E\to W$ with hermitean connection $\nabla_E$, and $D$ is constructed from $\nabla_S\otimes \nabla_E$ then $\alpha_D=\ch[\nabla_E]\wedge Td(\nabla_W)$. We often abuse the notation by letting $\ind_{APS}(D)$ denote $\ind(D_{APS}^+)$.

In fact, whenever $P$ is an order $0$ pseudo-differential projection on $L^2(\partial W,S_\partial)$, with $P-P_\partial$ being a smoothing operator, the boundary condition \eqref{APScondition} leads to a well-defined Fredholm operator, $D_P^+$. The index of $D_P^+$ is given by 
\begin{equation}
\label{apsdifferentp}
\ind (D_P^+)=\ind (D_{APS}^+)+\ind(P,P_\partial),
\end{equation}
where $\ind(P,P_\partial)$ is the index of the Fredholm pair of projections $(P,P_\partial)$. We sometimes write $\ind_{APS}(D,P)$ for $\ind(D^+_P)$. See more in \cite[Section I.10]{BossB}. A highly useful fact in regards to ``Atiyah-Patodi-Singer indices" is that they have well understood gluing properties. More precisely, we have the following proposition (whose proof follows from the Atiyah-Patodi-Singer index formula \eqref{APSformula}).
\begin{prop}
\label{gluap}
Assume that $Z$ is a compact manifold with boundary, $S\to Z$ is a Clifford bundle and $D^Z$ is a Dirac operator on $S$. Let $Y\subseteq Z^\circ$ be a closed hyper surface such that $Z=W_1\cup_Y W_2$ where $W_1$ and $W_2$ are compact manifolds with boundaries $\partial W_i=(-1)^iY\dot{\cup} M_i$, for a closed manifold $M_i$. Then,
\[\ind(D^Z_{APS})=\ind (D^{W_1}_{APS}) +\ind(D^{W_2}_{APS})+\dim \ker D^Y,\]
where $D^{W_i}$ is the Dirac operator on $S|_{W_i}$ given by restricting $D^Z$ to $W_i$ and $D^Y$ is the boundary operator on $Y$.
\end{prop}

An invariant closely related to the Atiyah-Patodi-Singer index theorem is the relative $\eta$-invariant. They were first studied in \cite{APS2}. We now assume that $M$ is a closed Riemannan oriented manifold, that $S\to M$ is a Clifford bundle equipped with a Hermitean Clifford connection $\nabla_S$ and that $F_1,F_2\to M$ are two flat Hermitean vector bundles of the same rank. Equipping $F_1$ and $F_2$ with their flat connections we can associate Dirac operators $D_{S,1}$ and $D_{S,2}$ on $S\otimes F_1$ respectively $S\otimes F_2$. The relative $\eta$-invariant, associated with this data, is defined to be the real number
\begin{equation}
\label{firstreletadef}
\rho(M,S,\nabla_S,F_1,F_2):=\frac{1}{2}\bigg(\eta_{D_{S,1}}(0)+\dim \ker D_{S,1}-\eta_{D_{S,2}}(0)-\dim \ker D_{S,2}\bigg).
\end{equation}
By \cite[Theorem $3.3$]{APS2}, the invariant $\rho(M,S,\nabla_S,F_1,F_2)\!\!\!\mod\! \Z$ is a bordism invariant. This fact is a direct consequence of \eqref{APSformula}; the relative $\eta$-invariant can jump only by the integer coming from an Atiyah-Patodi-Singer index of the bordism. We will revisit this result below in Proposition \ref{sesfortherelativegroups}. 

The relative invariant, without reducing modulo $\Z$, does not behave as well under bordism. However, \cite[Proposition $6.6$]{HReta} implies that relative $\eta$-invariant respects vector bundle modification (see Definition \ref{vectorbundlemodificationofrelativecycles} below or \cite[Definition 3.9]{HReta}). We use the notation $1_\field{R}\to M$ for the trivial real line bundle.

\begin{prop} (see \cite[Proposition 6.6]{HReta}) \\
Assume that $M$ is an oriented closed Riemannian manifold, $S\to M$ is a Clifford bundle equipped with a Hermitean Clifford connection $\nabla_S$, that $F_1,F_2\to M$ are two flat Hermitean vector bundles of the same rank and that $V\to M$ is an oriented even-dimensional vector bundle. We define $\pi_V:M^V:=S(V\oplus 1_\field{R})\to M$, the sphere bundle of $V\oplus 1_\field{R}$, and let $S^V\to M^V$ denote the vector bundle modification of $S$ along $V$ (see \cite[Definition $3.9$]{HReta} or below in Definition \ref{vectorbundlemodificationofrelativecycles}) equipped with its canonically associated Clifford connection $\nabla_{S^V}$. It holds that  
\[\rho(M,S,\nabla_S,F_1,F_2)=\rho(M^V,S^V,\nabla_{S^V},\pi_V^*F_1,\pi_V^*F_2).\]
\end{prop}

\subsection{Higher Atiyah-Patodi-Singer index theory}
\label{subsectionwithhigheraps}
To fully understand relative $\eta$-invariants, we will make use of higher Atiyah-Patodi-Singer theory. This theory was initiated in \cite{melrosepiazza} where an Atiyah-Patodi-Singer index theorem for families was proved. For families, it is not clear how to choose boundary conditions since the spectral projections $\chi_{[0,\infty)}(D_\partial)$ can vary in a non-continuous way in the family. A choice has to be made and it is clear from Equation \eqref{apsdifferentp} that this choice (in general) affects the end result. 

\begin{define}
\label{cliffordbundlesdefinition}
Assume that $C$ is a unital $C^*$-algebra and that $W$ is an oriented Riemannian manifold with boundary. We say that a smooth locally trivial bundle of finitely generated projective $C$-modules $\mathfrak{S}_{\C\ell\otimes C}\to W$ is a $C$-Clifford bundle if there is a Clifford bundle $S_{\C\ell}\to W$ and a smooth locally trivial bundle of finitely generated projective $C$-modules $\mathcal{E}_C\to W$ such that 
\[\mathfrak{S}_{\C\ell\otimes C}=S_{\C\ell}\otimes \mathcal{E}_C.\]
\end{define}

The assumption on $\mathfrak{S}_{\C\ell\otimes C}\to W$ to be smooth can be lifted using \cite[Theorem $3.14$]{Sch}. We note that it is also possible to equip any $C$-Clifford bundle with a $C$-valued hermitean inner product. A Clifford connection is a connection $$\nabla_\mathfrak{S}=\nabla_S\otimes \nabla_\mathcal{E}:=\nabla_S \otimes I + I\otimes \nabla_{\mathcal{E}}$$ 
on $\mathfrak{S}_{\C\ell\otimes C}$ such that $\nabla_S$ is a Clifford connection on $S_{\C\ell}$ and $\nabla_\mathcal{E}$ a connection on $\mathcal{E}_C$. After choosing a Clifford connection, we can construct a Dirac operator $D_\mathfrak{S}$ on $\mathfrak{S}_{\C\ell\otimes C}$. The construction of suitable Atiyah-Patodi-Singer boundary conditions on $D_\mathfrak{S}$ requires a bit of analysis. We recall the construction from \cite{LP} in the even-dimensional case. 

Assuming the Clifford connection $\nabla_\mathfrak{S}$ to be of product type near $\partial W$, we can as above find a $C$-Clifford bundle $\mathfrak{S}_\partial \to \partial W$ with a Dirac operator $D_{\mathfrak{S},\partial}$ and in a collar neighborhood near the boundary write 
\[D_\mathfrak{S}^+=\sigma\left(\frac{\partial}{\partial t}+D_{\mathfrak{S},\partial}\right),\]
for a unitary bundle isomorphism $\sigma:\mathfrak{S}^+|_{\partial W}\to \mathfrak{S}^-|_{\partial W}$. The Dirac operator $D_{\mathfrak{S},\partial}$ is an elliptic self-adjoint operator in the Mischenko-Fomenko pseudo-differential calculus of $C$-linear operators, see \cite{mischenkofomenko}; hence $D_{\mathfrak{S},\partial}$ forms a self-adjoint regular $C$-linear operator with $C$-compact resolvent. 

Recall the following terminology from \cite{LP}. A spectral cut (see \cite[Definition $2$]{LP}) is a function $\chi\in C^\infty(\R,[0,1])$ such that for some $a<b$ it holds that $\chi(t)=0$ for $t<a$ and $\chi(t)=1$ for $t>b$. We can in general consider a $C$-linear elliptic self-adjoint pseudo-differential operator $D$ in the Mischenko-Fomenko calculus on a smooth locally trivial bundle $\mathcal{E}\to M$ of finitely generated projective $C$-modules on a closed manifold $M$. A spectral section for $D$ (see \cite[Definition $3$]{LP} or \cite{wuwho}) is a projection $P\in \mathrm{End}_C^*(L^2(M,\mathcal{E}))$ such that there are two spectral cuts $\chi_1\leq \chi_2$ satisfying that $\mathrm{im}\,\chi_1(D)\subseteq \mathrm{im} \,P\subseteq \mathrm{im} \,\chi_2(D)$. We recall the following important Theorem from \cite{LP}:

\begin{theorem}[Theorem $3$ of \cite{LP}] 
\label{leichpiazzass}
There exists a spectral section for $D$ if and only if $\ind_C(D)=0\in K_1(C)$.
\end{theorem}

\begin{remark}
\label{reallysmoothremark}
Let $\Psi^{-\infty}_C(M,\mathcal{E})$ denote the algebra of smoothing operators in the Mischenko-Fomenko calculus on the $C$-bundle $\mathcal{E}$.  By \cite[Proposition $2.10$]{LPGAFA}, we can choose $P$ such that there is a self-adjoint smoothing operator $A\in \Psi^{-\infty}_C(M,\mathcal{E})$ with $D+A$ is invertible and
$$P=\chi_{[0,\infty)}(D+A)=\frac{1}{2}\left(\frac{D+A}{|D+A|}+1\right).$$
\end{remark}

Returning to the situation of a $C$-Clifford bundle $\mathfrak{S}\to W$ on an oriented manifold with boundary, the boundary operator $D_{\mathfrak{S},\partial}$ defines a class $\ind_C(D_{\mathfrak{S},\partial})\in K_{1}(C)$ which by (for example) \cite[Theorem $6.2$]{hilsumbordism} vanishes (i.e., $\ind_C(D_{\mathfrak{S},\partial})=0$). We summarize these conclusions in the following Lemma; we will make use of this result throughout the paper.

\begin{lemma}
Assume that $W$ is an even-dimensional oriented Riemannian manifold, $C$ is a unital $C^*$-algebra, $\mathfrak{S}\to W$ a $C$-Clifford bundle with Clifford connection and all of these geometric structures are of product type near $\partial W$. Then there is a smoothing operator $A\in \Psi^{-\infty}(\partial W,\mathfrak{S}_\partial)$ such that $D_{\mathfrak{S},\partial}+A$ is invertible and the projection $P_A:=\chi_{[0,\infty)}(D+A)$ defines a boundary condition making $D_\mathfrak{S}$ into a $C$-Fredholm operator with a well defined index class
\[\ind_{APS}(D_\mathfrak{S},A)\in K_{0}(C).\]
\end{lemma}

The proof of the final part of the Lemma can be found in \cite[Section 3]{LP}. The reader should recall from \eqref{apsdifferentp} that already when $C=\mathbb{C}$ the class $\ind_{APS}(D_\mathfrak{S},A)$ depends on the choice of $A$.

Higher Atiyah-Patodi-Singer theory comes with good functoriality properties. Assume that $\phi:C\to C'$ is a unital $*$-homomorphism; we also take data $(W,\mathfrak{S},\nabla_{\mathfrak{S}}, A)$ as above. The data can be pushed forward along $\phi$ to data $(W,\phi_*\mathfrak{S},\phi_*\nabla_{\mathfrak{S}}, \phi_*A)$ using the internal tensor product over $C$. By \cite[Corollary $C.2$]{PS}, $D_{\phi_*\mathfrak{S},\partial}+\phi_*A$ is invertible. Thus the class $\ind_{APS}(D_{\phi_*\mathfrak{S}},\phi_*A)\in K_{0}(C')$ is well defined. Moreover,
\begin{equation}
\label{functorialityofapsclasses}
\phi_* \ind_{APS}(D_\mathfrak{S},A)=\ind_{APS}(D_{\phi_*\mathfrak{S}},\phi_*A).
\end{equation}

We will need to express relative $\eta$-invariants through higher Atiyah-Patodi-Singer theory. Let $\Gamma$ be a finitely generated discrete group and assume that $M$ is a closed oriented Riemannian manifold equipped with a continuous map $f:M\to B\Gamma$. We let $\mathcal{L}:=E\Gamma\times_\Gamma C^*_{{\bf full}}(\Gamma)\to B\Gamma$ denote the Mischenko bundle for the full $C^*$-algebra of $\Gamma$ on $B\Gamma$. Whenever $\sigma_1,\sigma_2:\Gamma\to U(k)$ are two representations of $\Gamma$ we can define the vector bundles
\[E_i:=E\Gamma\times_{\sigma_i} \C^k=\mathcal{L}\times_{\sigma_i}\C^k\to B\Gamma,\quad i=1,2.\]
The vector bundles $F_i:=f^*E_i\to M$ are flat Hermitean vector bundles. For any Clifford bundle $S\to M$ with a Dirac operator $D_S$ associated with a Clifford connection $\nabla_S$, the relative $\eta$-invariant is defined as in \eqref{firstreletadef}:
\[\rho_{\sigma_1,\sigma_2}(M,S,f,D_S):=\frac{1}{2}\bigg(\eta_{D_{S,1}}(0)+\dim \ker D_{S,1}-\eta_{D_{S,2}}(0)-\dim \ker D_{S,2}\bigg),\]
where $D_{S,i}$ denotes the Dirac operator on $S\otimes F_i$ constructed from $D_S$ and the flat connection on $F_i$. 

\begin{notation}
\label{signarho}
When $M$ is an oriented manifold, $f:M\to B\Gamma$ is a continuous map, $S=\wedge^*_\C T^*M$, $D_{sign}$ is the signature operator (see \cite{APS1}) and $\sigma:\Gamma\to U(k)$ is a representation, the invariant 
\[\rho_\sigma(M,f):=\rho_{\sigma,k\varepsilon}(M,S,f,D_{sign}),\]
where $\varepsilon$ denotes the trivial representation, is referred to as the relative $\eta$-invariant for the signature operator.
\end{notation}

A \emph{stable relative $\eta$-invariant} can also be constructed (see \cite{PS} for more details). This invariant can only be defined if we make the following assumptions. We can lift $D_S$ to a Dirac operator $D_{S,\mathcal{L}}$ on $S\otimes f^*\mathcal{L}$ via the flat connection on $f^*\mathcal{L}$. The operator $D_{S,\mathcal{L}}$ is an elliptic $C^*_{{\bf full}}(\Gamma)$-linear self-adjoint operator and there is a well defined index class $\ind_{C^*_{{\bf full}}(\Gamma)}(D_{S,\mathcal{L}})\in K_{1}(C^*_{{\bf full}}(\Gamma))$. Assume that $\ind_{C^*_{{\bf full}}(\Gamma)}(D_{S,\mathcal{L}})=0$. By Theorem \ref{leichpiazzass} and \cite[Corollary $C.2$]{PS}, there exists an $A\in \Psi^{-\infty}_{C^*_{\bf full}(\Gamma)}(M,S\otimes f^*\mathcal{L})$ such that $D_{S,\mathcal{L}}+A$ is invertible, in particular $D_{S,i}+(\sigma_i)_*A$ is invertible for $i=1,2$. We define the stable relative $\eta$-invariant as
\[\rho_{\sigma_1,\sigma_2}^s(M,S,f,D_S):=\frac{1}{2}\bigg(\eta_{D_{S,1}+(\sigma_1)_*A}(0)-\eta_{D_{S,2}+(\sigma_2)_*A}(0)\bigg).\]

\begin{lemma}
\label{stablerhoinvariantlemma}
Assume that 
\begin{enumerate}
\item The full assembly for free actions\footnote{For definitions see below in Subsection \ref{sectionorientedmodel}.} $\mu:K_*(B\Gamma)\to K_*(C^*_{{\bf full}}(\Gamma))$ is surjective.
\item For the data $(M,S,f,D_S)$, we have that $\ind_{C^*_{{\bf full}}(\Gamma)}(D_{S,\mathcal{L}})= \mu(M,S,f)=0$. 
\end{enumerate}
Then the stable relative $\eta$-invariant $\rho_{\sigma_1,\sigma_2}^s(M,S,f,D_S)\in \R$ is well defined and does not depend on the choice of $A$.
\end{lemma}

The proof of Lemma \ref{stablerhoinvariantlemma} is by no means trivial; the reader is referred to \cite[Lemma $4.6$]{PS}. In \cite{PS}, under the stronger assumption of $\mu$ being an isomorphism, Piazza and Schick prove that $\rho_{\sigma_1,\sigma_2}^s(M,S,f,D_S)=0$ for any $A$. A geometric proof of this fact using the ``Baum approach to index theory" is given in Theorem  \ref{vanishingstablerho}.

\subsection{Baum-Douglas models}
\label{sectionorientedmodel}

The geometric model of \cite{paperI} is defined using spin$^c$-manifolds. However, as we saw in the previous section (also see \cite{HReta, Kes}), relative $\eta$-invariants are interesting in the larger generality of oriented manifolds. As such, the Baum-Douglas type geometric group constructed in \cite{HReta} encoding relative $\eta$-invariants uses oriented manifolds and Clifford bundles. This development is very much based on the papers \cite{guentnerkhom} and \cite{Kescont}. 

 In this section, the definitions and main techniques of ``oriented" Baum-Douglas type models from \cite{Kescont} are recalled and adapted to the case of coefficients in a Banach algebra. 

\begin{define}
\label{orientedcyclesdefinition}
Let $X$ be a locally compact Hausdorff space and $B$ a unital Banach algebra. An oriented $K$-cycle with coefficients in $B$ is a triple $(M,\mathfrak{S}_{\C\ell\otimes B},f)$ where:
\begin{enumerate}
\item $M$ is an oriented smooth manifold.
\item $\mathfrak{S}_{\C\ell\otimes B}\to M$ is a $B$-Clifford bundle (see Definition \ref{cliffordbundlesdefinition}).
\item $f:M\to X$ is a continuous map.
\end{enumerate}
\end{define}

The assumptions on $\mathfrak{S}_{\C\ell\otimes B}$ to be smooth can be lifted using \cite[Theorem $3.14$]{Sch}. Two oriented $K$-cycles with coefficients in $B$, $(M,\mathfrak{S}_{\C\ell\otimes B},f)$ and $(M',\mathfrak{S}'_{\C\ell\otimes B},f')$ are said to be isomorphic if there is a diffeomorphism $\kappa:M\to M'$ such that $f=f'\circ \kappa$ and $\mathfrak{S}_{\C\ell\otimes B}\cong \kappa^*\mathfrak{S}_{\C\ell\otimes B}'$ as bundles of Clifford $B$-bundles. The inverse of a cycle $(M,\mathfrak{S}_{\C\ell\otimes B},f)$ is defined as
\begin{equation}
\label{inversecycle}
-(M,\mathfrak{S}_{\C\ell\otimes B},f):=
\begin{cases}
(-M,\mathfrak{S}_{\C\ell\otimes B},f)\quad\mbox{if}\;\, M\;\,\mbox{is even-dimensional},\\
(-M,-\mathfrak{S}_{\C\ell\otimes B},f)\quad\mbox{if}\;\, M\;\,\mbox{is odd-dimensional}
\end{cases}
\end{equation}
Here $-M$ denotes the manifold $M$ with reversed orientation and $-\mathfrak{S}_{\C\ell\otimes B}$ the $B$-Clifford bundle $\mathfrak{S}_{\C\ell\otimes B}$ equipped with the opposite Clifford multiplication; the reader can find more details in \cite[Subsection $2.2$]{Kescont}.

We now turn to a construction related to the vector bundle modification of oriented cycles. Assume that $\pi:V\to M$ is an oriented vector bundle equipped with a metric. We let $1_\field{R}\to M$ denote the trivial real line bundle and form the sphere bundle 
\begin{equation}
\label{vecmodifiedmanifold}
\pi_V:M^V:=S(V\oplus 1_\field{R})\to M,
\end{equation}
whose fibers are spheres of the same dimension as the rank of $V$. There is an inclusion $V\hookrightarrow S(V\oplus 1_\field{R})$ such that the projection $\pi_V$ induces a diffeomorphism $S(V\oplus 1_\field{R})\setminus V\cong M$. 

The sphere $S(V\oplus 1_\field{R})$ can be realized as $S(V\oplus 1_\field{R})=S_+\cup_{S(V)}S_-$, where $S_\pm\to M$ are bundles of open balls (each of which is diffeomorphic to $V$). The metric on $V$ induces a splitting $TS_\pm \cong \pi_V^*V\oplus \pi_V^*TM$ which glues together (globally) to give a decomposition $TS(V\oplus 1_\field{R})\cong T_{ver}\oplus T_{hor}$; it also follows that $T_{ver}|_V\cong \pi_V^*V$. We define the graded $T_{ver}$-Clifford bundle $\hat{V}\to S(V\oplus 1_\field{R})$ as 
\[\hat{V}:=\ker(\epsilon_{dR}\epsilon_{H}- 1)\subseteq \wedge^*_\C T_{ver},\]
where $\epsilon_{dR}$ is the ordinary grading by form degree on $\wedge^*_\C T_{ver}$ and $\epsilon_{H}$ is the grading constructed from the Hodge $*$ on $\wedge^*_\C T_{ver}$; the grading on $\hat{V}$ is given by $\epsilon_{dR}$.

\begin{define}
\label{thombundledef}
The vector bundle modification of an oriented cycle $(M,\mathfrak{S}_{\C\ell \otimes B},f)$ for $K_*^h(X;B)$ along an even-dimensional oriented vector bundle $V\to M$ is the cycle
\[(M,\mathfrak{S}_{\C\ell \otimes B},f)^V:=(M^V,\pi^*_V\mathfrak{S}_{\C\ell\otimes B}\hat{\otimes} \hat{V},f\circ \pi_V).\]
\end{define}

The next Proposition guarantees that vector bundle modification in the oriented model is compatible with the spin$^c$-version. For a proof, see \cite[Page $61$]{Kescont}. 

\begin{prop}
\label{tauvversusbott}
Whenever $V\to Z$ admits a spin$^c$-structure $S_V\to Z$, there is a natural isomorphism of $T^*_{ver}$-Clifford bundles $\hat{V}\cong \pi_V^*S_V\otimes (\pi_V^*S_V)^{+*}$.
\end{prop}

The set of isomorphism classes of oriented $K$-cycles with coefficients in $B$ can be equipped with an equivalence relation generated by the usual relations disjoint union/direct sum, bordism and vector bundle modification. These notions are defined for $B=\C$ in \cite[Subsection $2.2$]{Kescont} and generalizes mutatis mutandis to the general case. We let $K_*^h(X;B)$ denote the quotient of the set of isomorphism classes of oriented $K$-cycles with coefficients in $B$ by this relation. The set $K_*^h(X;B)$ forms an abelian group under disjoint union while the inverse of a cycle defined in Equation \eqref{inversecycle} forms an inverse under disjoint union in this set.

\begin{remark}
In the case that $B=\C$, we use the notation $K_*^h(X):=K_*^h(X;\C)$. It is obvious from the definitions that this group coincides with the group defined in \cite{Kescont}.
\end{remark}

In \cite{paperI}, $K_*^{geo}(X;B)$ was defined using cycles that are triples $(M,\mathcal{E}_B,f)$ where $M$ is spin$^c$ and $\mathcal{E}_B\to M$ is a smooth locally trivial bundle of finitely generated projective $B$-modules; we refer to these cycles as $K$-cycles with coefficients in $B$. In analogy with \cite[Lemma $2.8$]{Kescont}, we have the following result:

\begin{lemma}
\label{khomorientedlem}
The map from the set of $K$-cycles with coefficients in $B$ to the set of oriented $K$-cycles with coefficients in $B$ given by 
\[(M,\mathcal{E}_B,f)\mapsto (M,S_M\otimes \mathcal{E}_B,f),\]
where $S_M\to M$ is the spin$^c$-structure of $M$, defines an isomorphism
\[K_*^{geo}(X;B)\cong K_*^h(X;B).\]
In particular, there is a natural isomorphism $K_*^h(pt;B)\cong K_*(B)$.
\end{lemma}

\begin{define}
\label{orientedassemblydefinition}
Assume that $B$ is a unital Banach algebra, $X$ a locally compact Hausdorff space and $\mathcal{L}_B\to X$ a locally trivial bundle of finitely generated projective $B$-modules. The oriented assembly map $\mu_\mathcal{L}^h:K_*^h(X)\to K_*^h(pt;B)$ along $\mathcal{L}_B$ is given on cycles by
\[(M,S_{\C\ell},f)\mapsto (M,S_{\C\ell}\otimes f^*\mathcal{L}_B).\]
\end{define}

\begin{remark}
In \cite[Definition $1.3$]{paperI}, the assembly map in the spin$^c$-model is defined at the level of cycles via $\mu_\mathcal{L}:(M,E_\C,f)\mapsto(M,E_{\C}\otimes f^*\mathcal{L}_B)$. It is clear from the definitions involved (i.e.,  \ref{orientedassemblydefinition} and \cite[Definition $1.3$]{paperI}) that the following diagram commutes 
\begin{center}
$\begin{CD}
K_*^{geo}(X)@>\mu_\mathcal{L}>> K_*^{geo}(pt;B)   \\
 @VVV  @VVV  \\
K_*^h(X)@>\mu_\mathcal{L}^h>>K_*^h(pt;B)
\end{CD}$
\end{center}
where the vertical arrows are the natural isomorphisms discussed in Lemma \ref{khomorientedlem}.
\end{remark}

The prototypical example of assembly (as defined in Definition \ref{orientedassemblydefinition}) is the assembly map for free actions; this assembly map plays a key role in this paper. Let $\Gamma$ be is a finitely generated discrete group, $E\Gamma\to B\Gamma$ be the universal $\Gamma$-bundle over the classifying space, $B\Gamma$, of $\Gamma$, and $\mathcal{A}(\Gamma)$ be a Banach algebra completion of $\C[\Gamma]$. The completion most relevant in this paper is the full $C^*$-completion. However, other completions are possible (e.g., the reduced $C^*$-completion $C^*_{\bf red}(\Gamma)$ or $\ell^1(\Gamma)$). The Mischenko bundle is given by
\[\mathcal{L}_\mathcal{A}=E\Gamma\times_\Gamma \mathcal{A}(\Gamma)\to B\Gamma.\]
If $\Gamma$ is torsion-free, $\mathcal{A}(\Gamma)$ is said to have the Baum-Connes property if $\mu_{\mathcal{L}_\mathcal{A}}$ is an isomorphism. If $C^*_{\bf full}(\Gamma)$ has the Baum-Connes property, we say that $\Gamma$ has the full Baum-Connes property.

The assembly along $\mathcal{L}_{C^*_{\bf red}}$ defines a map $\mu_{\bf red}:K_*(B\Gamma)\to K_*(C^*_{\bf red}(\Gamma))$ that is often referred to as the reduced assembly for free actions. If $\Gamma$ is torsion-free, then $\mu_{\bf red}$ is equivalent to the Baum-Connes assembly map; the well known conjecture of Baum and Connes predicts that any torsion-free discrete group has the reduced Baum-Connes property, i.e., $\mu_{\bf red}$ is an isomorphism. We refer the reader to the discussion in \cite[Section $1$]{paperI} for more details.

\section{The relative group of Higson-Roe}
\label{higsonroesgroup}

We now turn to the relation between relative $\eta$-invariants and geometric $K$-homology. As in Subsection \ref{subsectionwithhigheraps}, we let $\sigma_1$ and $\sigma_2$ denote two rank $k$ unitary representations of a discrete group $\Gamma$. In \cite{HReta}, a group $\mathcal{S}^{h}_1(\sigma_1,\sigma_2)$ was constructed as a natural domain for the relative $\eta$-invariant. We will recall the construction of this group below. Before recalling the geometric group $\mathcal{S}^{h}_1(\sigma_1,\sigma_2)$ of Higson and Roe, we introduce further notation regarding Dirac operators. 

\begin{notation}[Notation for Dirac operators]
Let $M$ be a closed oriented manifold and $f:M\to B\Gamma$ a continuous map. If $S_{\C\ell}\to M$ is a Clifford bundle with Clifford connection $\nabla_{S_{\C\ell}}$, we let $D^M_{\nabla_{S_{\C\ell}}}$ denote the associated Dirac operator. We also let $-S_{\C\ell}$ denote the Clifford bundle $S_{\C\ell}$ equipped with the opposite Clifford multiplication; its Dirac operator equals $-D^M_{\nabla_{S_{\C\ell}}}$. When the choice of data is understood from the context, we denote Dirac operators by $D^M_{S_{\C\ell}}$, $D_S$ or $D^M$ or sometimes simply by $D$. If $S=S_M\otimes E$ for a vector bundle $E$ and a spin$^c$-structure $S_M$ we sometimes also write $D_E$. Letting  $\sigma_1$ and $\sigma_2$  be as above, 
$$E_i:=E\Gamma\times_{\sigma_i} \C^k\to B\Gamma.$$
If $D$ is any Dirac type operator on $M$ acting on a Clifford bundle $S_{\C\ell}$, we let $D_i$, or sometimes $D^M_{S_{\C\ell}\otimes E_i}$ if the Dirac operator is constructed as above, denote the Dirac operator on $S_{\C\ell}\otimes f^*E_i$ given by twisting $D$ by the flat connection on $f^*E_i$. 

Whenever no index is specified, $\mathcal{L}$ denotes the Mischenko bundle for the full $C^*$-completion of $\C[\Gamma]$. We will reserve $A$ for smoothing operators $A\in \Psi^{-\infty}_{C^*_{{\bf full}}(\Gamma)}(M,S_{\C\ell}\otimes f^*\mathcal{L})$ in the Mischenko-Fomenko calculus. We use the notation $A_i:=(\sigma_i)_*(A)\in \Psi^{-\infty}(M;S_{\C\ell}\otimes f^*E_i)$.
\end{notation}

\subsection{The relative group of Higson-Roe}

We recall the relative group of Higson-Roe as constructed in \cite[Section 8]{HReta}, but ``decorate" cycles with a smoothing operator on the Mischenko bundle for the full $C^*$-completion. This does not alter the cycles considerably after imposing the bordism relation; however, it makes the study of the stable relative $\eta$-invariants fit more naturally into the framework of \cite{HReta}.

\begin{define}
A collection $(M,S_{\C\ell},f,D,A,n)$ is called a decorated odd oriented $(\sigma_1,\sigma_2)$-cycle if
\begin{enumerate}
\item $(M,S_{\C\ell},f)$ is a cycle for $K_1^h(B\Gamma)$;
\item $D$ is a specific choice of Dirac operator for $(M,S_{\C\ell},f)$;
\item $A\in \Psi^{-\infty}_{C^*_{\bf full}(\Gamma)}(M, S_{\C\ell}\otimes f^*\mathcal{L})$ is a self-adjoint smoothing operator;
\item $n$ is an integer;
\end{enumerate}  
The inverse of an odd oriented $(\sigma_1,\sigma_2)$-cycle $(M,S_{\C\ell},f,D,A,n)$ as
\begin{equation}
\label{inversionofrelcycles}
-(M,S_{\C\ell},f,D,A,n):=(-M,-S_{\C\ell},f,-D,-A,d_2(D,A)-d_1(D,A)-n),
\end{equation}
where 
\begin{equation}
\label{dinot}
d_i(D,A):=\dim\ker (D_i+A_i). 
\end{equation}
A collection of the form $(M,S_{\C\ell},f,D,n)$ such that $(M,S_{\C\ell},f,D,0,n)$ is a decorated odd oriented $(\sigma_1,\sigma_2)$-cycle is called an odd oriented $(\sigma_1,\sigma_2)$-cycle.
\end{define}

The notion of bordism for (decorated) odd oriented $(\sigma_1,\sigma_2)$-cycles is slightly involved. First we introduce the notion of a $(\sigma_1,\sigma_2)$-cycle with boundary.

\begin{define}
\label{apicyclewithboundary}
A decorated $(\sigma_1,\sigma_2)$-cycle with boundary is a quintuple $(W,S_{\C\ell}^W,g,Q,A)$ where 
\begin{enumerate}
\item $(W,S_{\C\ell}^W,g)$ is an even oriented $K$-cycle with boundary of product type $S_{\C\ell}^W\cong S_{\C\ell}^{\partial W}\hat{\otimes} S_N$ near the boundary.
\item $Q$ is a choice of Dirac operator for $(W,S_{\C\ell}^W,g)$ that is of product type $Q=\sigma(\frac{\partial}{\partial t}+Q_\partial)$ near the boundary (cf. Equation \eqref{diracofproducttupe} on page \pageref{diracofproducttupe}).
\item $A\in \Psi^{-\infty}_{C^*_{\bf full}(\Gamma)}(\partial W,S_{\C\ell}^{\partial W}\otimes g^*\mathcal{L})$ is self-adjoint.
\end{enumerate}
If $A=0$, we say that the quadruple $(W,S_{\C\ell}^W,g,Q)$ is a $(\sigma_1,\sigma_2)$-cycle with boundary.
\end{define}

With a $(\sigma_1,\sigma_2)$-cycle with boundary $(W,S_{\C\ell}^W,g,Q,A)$ there are two associated Atiyah-Patodi-Singer index problems; one is associated to $(Q_1,\chi_{[0,\infty)}(Q_\partial+A_1))$ and the other is associated to $(Q_2,\chi_{[0,\infty)}(Q_\partial+A_2))$.  We will use the notation $\ind_{APS}(Q,A)_i:=\ind_{APS}(Q_i,\chi_{[0,\infty)}(Q_{\partial,i}+A_i))$. It follows from Lemma \ref{stablerhoinvariantlemma} that 
\begin{align}
\nonumber
\frac{1}{2}\bigg(\eta_{Q_{\partial,1}+A_1}(0)+d_1(Q_{\partial},A)\bigg)-\frac{1}{2}&\bigg(\eta_{Q_{\partial,2}+A_2}(0)+d_2(Q_{\partial},A)\bigg)=\\
&\ind_{APS}(Q,A)_2-\ind_{APS}(Q,A)_1\in \Z.
\label{apsdifferenceformula}
\end{align}
Morevoer, if $Q_{\partial,\mathcal{L}}+A$ is invertible, Equation \eqref{functorialityofapsclasses} on page \pageref{functorialityofapsclasses}, implies that
$$\ind_{APS}(Q,A)_i=(\sigma_i)_*\ind_{APS}(Q_\mathcal{L},A).$$

\begin{define}
\label{bordismofdecorated}
The boundary of $(W,S_{\C\ell}^W,g,Q,A)$ is
\[\partial(W,S_{\C\ell}^W,g,Q,A)=(\partial W,S_{\C\ell}^{\partial W},g|_{\partial W},Q_\partial,A,n),\]
where 
\[n:=\ind_{APS}(Q,A)_1-\ind_{APS}(Q,A)_2.\]
We say that two decorated odd oriented $(\sigma_1,\sigma_2)$-cycles $(M,S_{\C\ell},f,D,A,n)$ and $(M',S'_{\C\ell},f',D',A',n')$ are decoratedly bordant if 
$$(M,S_{\C\ell},f,D,A,n)\dot{\cup}-(M',S'_{\C\ell},f',D',A',n')$$ 
is the boundary of a decorated $(\sigma_1,\sigma_2)$-cycle with boundary. In the same way a $(\sigma_1,\sigma_2)$-cycle with boundary defines the notion of bordism on odd oriented $(\sigma_1,\sigma_2)$-cycles (see \cite[Definitions 8.9 and 8.13(b)]{HReta}).
\end{define}

If $(M,S_{\C\ell},f,D,A,n)$ is an odd oriented $(\sigma_1,\sigma_2)$-cycle we define
\begin{align}
\label{rhodefinition}
\nonumber
\rho_{\sigma_1,\sigma_2}&(M,S_{\C\ell},f,D,A):=\\
&\frac{1}{2}\bigg(\eta_{D_1+A_1}(0)+d_1(D,A)\bigg)-\frac{1}{2}\bigg(\eta_{D_2+A_2}(0)+d_2(D,A)\bigg)\in \R
\end{align}
On the right hand side, all of the data in $(M,S_{\C\ell},f,D,A)$ is used even though the notation supresses some of the data.

\begin{prop}
\label{bordingoutsmoothing}
There is a bordism 
$$(M,S_{\C\ell},f,D,A,n)\sim_{bor} (M,S_{\C\ell},f,D,0,n+m)$$ 
where 
\begin{align*}
m=\rho_{\sigma_1,\sigma_2}(M,S_{\C\ell},f,D,A) -\rho_{\sigma_1,\sigma_2}(M,S_{\C\ell},f,D,0).
\end{align*}
\end{prop}

\begin{proof}
Define $W:=M\times [0,1]$, $g:=f\circ \pi_M$ and $Q=\sigma\left(\frac{\partial}{\partial t}+D\right)$, where $t$ denotes the coordinate in $[0,1]$ and $\sigma$ is the bundle automorphism given by Clifford multiplication by $\mathrm{d} t$ (see more in Subsection \ref{apssubsub} and Equation \eqref{diracofproducttupe}). Define $\tilde{A}$ as $A$ on $M\times \{1\}$ and $0$ on $M\times \{0\}$. It follows from \eqref{apsdifferenceformula} that $m\in \Z$ and that $(M,S_{\C\ell},f,D,A,n)\dot{\cup}-(M,S_{\C\ell},f,D,0,n+m)$ is the boundary of $(W,S_{\C\ell}^W,g,Q,\tilde{A})$.
\end{proof}

\begin{remark}
\label{bordismremarks}
A consequence of Proposition \ref{bordingoutsmoothing} is that any odd oriented $(\sigma_1,\sigma_2)$-cycle $(M,S_{\C\ell},f,D,A_0,n)$ is bordant to a cycle of the form $(M,S_{\C\ell},f,D,0,n'')$. If $\mu_{C^*_{\bf full}}(M,S_{\C\ell},f)=0$, the cycle is bordant to a cycle of the form $(M,S_{\C\ell},f,D,A,n')$ where $D_\mathcal{L}+A$ is invertible. To simplify notation, we denote a cycle of the form, $(M,S_{\C\ell},f,D,0,n)$, by $(M,S_{\C\ell},f,D,n)$.
\end{remark}

A $2k$-dimensional oriented vector bundle $V\to M$ is uniquely determined by its principal $SO(2k)$-bundle $P_V\to M$ of oriented frames by $V=P_V\times_{SO(2k)}\field{R}^{2k}$. The sphere bundle $M^V=S(V\oplus 1_\field{R})\to M$ can be reconstructed from $P_V$ by $S(V\oplus 1_\field{R})=P_V\times_{SO(2k)}S^{2k}$. If a principal $SO(2k)$-bundle $P$ is given, we let $\pi_P:M^P:=P\times_{SO(2k)}S^{2k}\to M$. We follow the construction of vector bundle modification from \cite{HReta}. Let $D_\theta$ denote a $SO(2k)$-equivariant Dirac operator on $S^{2k}$ whose kernel is a one-dimensional copy of the trivial representation (see \cite[Lemma $6.7$]{HReta}). Whenever $P\to M$ is a principal $SO(2k)$-bundle, it is possible to identify
\begin{equation}
\label{ltwoonprincipalbundle}
L^2(M^P;\pi_P^*S\hat{\otimes}\hat{V})\cong L^2(P\times S^{2k};\pi^*S\hat{\boxtimes} \hat{V}_{S^{2k}})^{SO(2k)},
\end{equation}
which is the $SO(2k)$-invariant direct summand in $L^2(P\times S^{2k},\pi^*S\hat{\boxtimes} \hat{V}_{S^{2k}})$. Furthermore, whenever $D$ is an operator on $M$, it lifts to an equivariant operator on $P$ that we will denote by $\pi_P^*(D)$. The following definition is a minor variation of \cite[Definition $8.13$]{HReta} (also see \cite[Section 4.1]{DeeMapCone}).

\begin{define}
\label{vectorbundlemodificationofrelativecycles}
Assume that $(M,S_{\C\ell},f,D,A,n)$ is a decorated odd oriented $(\sigma_1,\sigma_2)$-cycle and that $P\to M$ is a smooth principal $SO(2k)$-bundle. We define the data
\begin{enumerate}
\item $D^P$ is defined as the restriction of the sum of graded tensor products $\pi_P^*(D)\hat{\otimes} \epsilon_{dR} +1_{L^2(P,\pi_P^*S)}\hat{\otimes} D_\theta$ to its maximal domain in $L^2(M^P;\pi_P^*S\hat{\otimes}\hat{V})\subseteq L^2(P\times S^{2k};\pi^*S\hat{\boxtimes} \hat{V}_{S^{2k}})$, here $\epsilon_{dR}$ denotes the grading operator $L^2(S^{2k},\hat{V}_{S^{2k}})$ defined from form degree.
\item The smoothing operator $A^P$ is defined as the restriction of $\pi_P^*(A)\hat{\otimes} e_\theta$ to $L^2(M^P,\pi_P^*S\otimes f^*\mathcal{L}\hat{\otimes}\hat{V}_P)$, where $e_\theta$ is the even smoothing operator on $S^{2k}$ giving the projection onto the kernel of $D_\theta$.
\item $\hat{V}_P$ is the Thom bundle constructed from the oriented vector bundle $P\times_{SO(2k)} \field{R}^{2k}\to M$, defined as in Definition \ref{thombundledef}.
\end{enumerate}
The vector bundle modification of $(M,S_{\C\ell},f,D,A,n)$ along $P$ is defined to be the odd oriented $(\sigma_1,\sigma_2)$-cycle
\[(M,S_{\C\ell},f,D,A,n)^P:=(M^P,\pi_P^*S_{\C\ell}\otimes \hat{V}_P,f\circ \pi_P,D^P,A^P,n).\]
\end{define}

\begin{prop}
The vector bundle modification of decorated odd oriented $(\sigma_1,\sigma_2)$-cycles in Definition \ref{vectorbundlemodificationofrelativecycles} is well defined.
\end{prop}

\begin{proof}
Using the constructions of \cite{HReta}, it remains to prove that $A^P$ is a well defined smoothing operator. This is clear since the integral kernel of $A^P$ is restriction of the exterior product of the pullback of the integral kernel of $A$ along $M^P\times M^P\to M\times M$ with the projection onto the kernel of $D_\theta$ to $L^2(M^P)$.
\end{proof}

The next Lemma shows that our notion of vector bundle modification is well behaved  with respect to higher Atiyah-Patodi-Singer theory (compare with \cite[Section 4.1]{DeeMapCone}).

\begin{lemma}
Let $(M,S_{\C\ell},f,D,A,n)$ be a decorated odd oriented $(\sigma_1,\sigma_2)$-cycle such that the $C^*_{\bf full}(\Gamma)$-linear operator $D_\mathcal{L}+A$ is invertible and let $P\to M$ be a principal $SO(2k)$-bundle. Then the $C^*_{\bf full}(\Gamma)$-linear operator $D^P_\mathcal{L}+A^P$ is invertible.
\end{lemma}

\begin{proof}
On the complemented $C^*_{\bf full}(\Gamma)$-submodule 
\small
$$L^2(M,S\otimes f^*\mathcal{L})=(1_{L^2(M,S\otimes f^*\mathcal{L})}\otimes e_\theta)\cdot L^2(M^P,\pi_P^*S\otimes f^*\mathcal{L}\hat{\otimes}\hat{V}_P)\subseteq L^2(M^P,\pi_P^*S\otimes f^*\mathcal{L}\hat{\otimes}\hat{V}_P),$$ 
\normalsize
it holds that 
$$(D^P_\mathcal{L}+A^P)|_{L^2(M,S\otimes f^*\mathcal{L})}=D_\mathcal{L}+A.$$
Hence the restriction of $D^P_\mathcal{L}+A^P$ to the complemented $C^*_{\bf full}(\Gamma)$-submodule $L^2(M,S\otimes f^*\mathcal{L})$ is invertible. On the complement of this $C^*_{\bf full}(\Gamma)$-submodule; 
\small
\begin{align*}
L^2(M,&\,S\otimes f^*\mathcal{L})^\perp\\
&=(\id-1_{L^2(M,S\otimes f^*\mathcal{L})}\otimes e_\theta)\cdot L^2(M^P,\pi_P^*S\otimes f^*\mathcal{L}\hat{\otimes}\hat{V}_P)\subseteq L^2(M^P,\pi_P^*S\otimes f^*\mathcal{L}\hat{\otimes}\hat{V}_P),
\end{align*}
\normalsize
or rather in the core 
$$(\id-1_{L^2(M,S\otimes f^*\mathcal{L})}\otimes e_\theta)\cdot C^\infty(M^P,\pi_P^*S\otimes f^*\mathcal{L}\hat{\otimes}\hat{V}_P)\subseteq L^2(M,\,S\otimes f^*\mathcal{L})^\perp,$$ 
it holds that
\begin{align*}
\big((D^P_\mathcal{L}+A^P)&|_{L^2(M,S\otimes f^*\mathcal{L})^\perp}\big)^2\\
&=\left(\pi_P^*(D_\mathcal{L})^2\hat{\otimes} (\id-e_\theta) +1_{L^2(P,\pi_P^*S\otimes f^*\mathcal{L})}\hat{\otimes} D_\theta\right)|_{L^2(M,S\otimes f^*\mathcal{L})^\perp}.
\end{align*}
Since $D_\theta^2$ is strictly positive on the image of $\id-e_\theta$, and $\pi_P^*(D_\mathcal{L})^2\hat{\otimes} (\id-e_\theta)\geq 0$, it follows that $(D^P_\mathcal{L}+A^P)|_{L^2(M,S\otimes f^*\mathcal{L})^\perp}$ is invertible.
\end{proof}

\begin{define}
\label{srelgroup}
We define $\tilde{\mathcal{S}}_1^h(\sigma_1,\sigma_2)$ to be the set of equivalence classes of decorated odd oriented $(\sigma_1,\sigma_2)$-cycles under the relation generated by decorated bordism, vector bundle modification and disjoint union/direct sum:
$$(M,S_{\C\ell},f,D,A,n)\dot{\cup}(M,S_{\C\ell}',f,D',A',n')\sim (M,S_{\C\ell}\oplus S_{\C\ell}',f,D\oplus D',A\oplus A',n+n').$$
We define $\mathcal{S}_1^h(\sigma_1,\sigma_2)$ to be the set of equivalence classes of odd oriented $(\sigma_1,\sigma_2)$-cycles under the relation generated by bordism, vector bundle modification and disjoint union/direct sum:
$$(M,S_{\C\ell},f,D,n)\dot{\cup}(M,S_{\C\ell}',f,D',n')\sim (M,S_{\C\ell}\oplus S_{\C\ell}',f,D\oplus D',n+n').$$
\end{define}

\begin{remark}
The group defined in \cite[Definition 8.13]{HReta} (which was denoted by $\mathcal{S}_1^{geom}(\sigma_1,\sigma_2)$ in \cite{HReta}) coincides with our group $\mathcal{S}_1^h(\sigma_1,\sigma_2)$ by construction. That the group also coincides with $\tilde{\mathcal{S}}_1^h(\sigma_1,\sigma_2)$ is a bit more subtle.
\end{remark}

\begin{lemma}
\label{relwwotilde}
The map defined at the level of cycles via
\[\mathcal{S}_1^h(\sigma_1,\sigma_2)\to \tilde{\mathcal{S}}_1^h(\sigma_1,\sigma_2), \quad (M,S_{\C\ell},f,D,n)\mapsto  (M,S_{\C\ell},f,D,0,n),\]
induces an isomorphism of groups.
\end{lemma}

\begin{proof}
It follows from Remark \ref{bordismremarks} that the map is a surjection. A splitting on the level of cycles for this map is given by
\[(M,S_{\C\ell},f,D,A,n)\mapsto (M,S_{\C\ell},f,D,n+m),\]
where the number $m$ is given in Proposition \ref{bordingoutsmoothing}. This map clearly respects both bordism and the disjoint union/direct sum relation. The proof of the Lemma is complete upon noting that \cite[Proposition $6.7$]{HReta} implies that this map on cycles respects vector bundle modification; thus, it induces a well defined inverse map $\tilde{\mathcal{S}}_1^h(\sigma_1,\sigma_2)\to \mathcal{S}_1^h(\sigma_1,\sigma_2)$.
\end{proof}
Following Higson and Roe \cite[Definition 8.15]{HReta}, we have the following definition. 
\begin{prop}
\label{definitionofrho}
The relative $\eta$-invariant $\rho_{\sigma_1,\sigma_2}: \tilde{\mathcal{S}}_1^h(\sigma_1,\sigma_2)\to \field{R}$, defined at the level of cycles via
\[\rho_{\sigma_1,\sigma_2}(M,S_{\C\ell},f,D,A,n):=n+\rho_{\sigma_1,\sigma_2}(M,S_{\C\ell},f,D,A),\]
is well defined. 
\end{prop}
\begin{proof}
It follows from Definition \ref{bordismofdecorated} that $\rho_{\sigma_1,\sigma_2}$ respects decorated bordism. Hence, by Proposition \ref{bordingoutsmoothing}, it suffices to prove that $\rho_{\sigma_1,\sigma_2}$ is well defined as a map $\mathcal{S}_1^h(\sigma_1,\sigma_2)\to \field{R}$; this was proved in \cite[Proposition $8.14$]{HReta}.
\end{proof}

We define $\delta_{\sigma_1,\sigma_2}:\tilde{\mathcal{S}}^{h}_1(\sigma_1,\sigma_2)\to K_1^h(B\Gamma)$ on cycles by
\[\delta_{\sigma_1,\sigma_2}(M,S_{\C\ell},f,D,A,n)=(M,S_{\C\ell},f).\]
We also define the $\rho$-invariant 
$$\bar{\rho}_{\sigma_1,\sigma_2}:K_1^h(B\Gamma)\to \field{R}/\Z, \quad \bar{\rho}_{\sigma_1,\sigma_2}:(M,S_{\C\ell},f)\mapsto \rho_{\sigma_1,\sigma_2}(M,S_{\C\ell},f,D,A)\,\mathrm{mod}\; \Z,$$
for any choice $D$ of a Dirac operator on $S_{\C\ell}$ and self-adjoint smoothing operator $A$. The map $\bar{\rho}_{\sigma_1,\sigma_2}$ is well defined by \cite[Theorem $6.1$]{HReta}, see also \cite[Theorem $3.3$]{APS2}. We note the following consequence of \cite[Lemma $8.24$]{HReta} and Proposition \ref{definitionofrho}.

\begin{prop}
\label{sesfortherelativegroups}
There is a commuting diagram with exact rows:
\[\begin{CD}
0@>>>\Z @>>> \tilde{\mathcal{S}}^{h}_1(\sigma_1,\sigma_2)@>\delta_{\sigma_1,\sigma_2}>> K_1^h(B\Gamma)@>>>0\\
@. @|  @V\rho_{\sigma_1,\sigma_2} VV @V\bar{\rho}_{\sigma_1,\sigma_2}VV  @. \\
0@>>>\Z @>>> \field{R}@>>> \field{R}/\Z@>>>0,
\end{CD}.\]
\end{prop}

\begin{remark}
The careful reader will note a difference in the sign conventions considered here and those used in \cite{HReta}. This manifests itself as follows: the roles of $\sigma_1$ and $\sigma_2$ are interchanged in the inversion formula for cycles in Equation \eqref{inversionofrelcycles}. This difference is explained by the use of the Atiyah-Patodi-Singer index formula \eqref{APSformula}  (on page \pageref{APSformula}) as stated in \cite{APS1}; the index formula for Dirac operators with the Atiyah-Patodi-Singer index formula in \cite[Theorem $2.2$]{HReta} contains a different sign. The difference in sign can be explained by the usage of the spectral projection onto the positive spectrum in contrast to the usage of the spectral projection onto the non-negative spectrum in \cite{APS1}.
\end{remark}

\begin{remark}
\label{ell2group}
The reader should note the following connection between the results in this section and results in \cite{BRvonN}. Namely, the techniques developed to this point can be applied in the construction considered in \cite[Section 4.2]{BRvonN}. In particular, in \cite[Definition 4.1]{BRvonN}, a geometric group is defined; its cycles take the form $(M,S,f,D,x)$, for $x\in \field{R}$. The only change required to construct this group is as follows: the bordism relation should be defined using the difference of the Atiyah-Patodi-Singer index and the $L^2$-Atiyah-Patodi-Singer index (rather than the $(\sigma_1,\sigma_2)$-relative Atiyah-Patodi-Singer index discussed in this section). The Cheeger-Gromov $\ell^2$-relative $\eta$-invariant (see \cite{PS} and \cite[Definition 5.1]{BRvonN}) is well defined on this group. As mentioned, this is very much in line with the development in \cite{BRvonN}--in particular, see \cite[Proposition 5.2]{BRvonN}.
\end{remark}

\section{The map $\Phi_\mathcal{A}$}
\label{mappingtohrrelgroup}

In this section, under suitable assumptions on the Banach algebra $\mathcal{A}(\Gamma)$, we construct a map $\Phi_\mathcal{A}:\mathcal{S}^{geo}_0(\Gamma,\mathcal{A})\to \mathcal{S}^{h}_1(\sigma_1,\sigma_2)$. In this paper, the applications of the mapping $\Phi_\mathcal{A}$ are (for the most part) concerned with the case when $\mathcal{A}$ is the full $C^*$-completion of $\C[\Gamma]$. 
 We will for the most of this section assume that $\mathcal{A}(\Gamma)$ is a Banach algebra completion of $\C[\Gamma]$ such that $\sigma_1$ and $\sigma_2$ extend continuously to $\mathcal{A}(\Gamma)$. We will use the notation and terminology of \cite{paperI}. To begin, we need the notion of a choice of Dirac operators on the $K$-theory cocycles.

\begin{define}[Dirac operators on $\mu_\mathcal{L}$-relative $K$-theory cocycles] 
\label{connectionsonrelcycles}
Assume that $W$ is an even-dimensional spin$^c$-manifold with boundary and $f:\partial W\to X$. Assume that $\xi=(\mathcal{E}_{\mathcal{A}(\Gamma)},\mathcal{E}^{\prime}_{\mathcal{A}(\Gamma)} , E_{\field{C}}, E_{\field{C}}^{\prime} ,\alpha )$ is a $\mu_\mathcal{L}$-relative $K$-theory cocycle on $(W,\partial W,f)$ (see \cite[Definition 1.8]{paperI}). Assume that $D_{\mathcal{E}}$ and $D_{\mathcal{E}'}$ are Dirac operators on $S_W\otimes \mathcal{E}_{\mathcal{A}(\Gamma)}$ respectively $S_W\otimes \mathcal{E}^{\prime}_{\mathcal{A}(\Gamma)}$ defined from $\mathcal{A}(\Gamma)$-Clifford connections of product type near $\partial W$. Assume that $D_{E}$ and $D_{E'}$ are Dirac operators on $S_{\partial W}\otimes E_{\field{C}}$ respectively $S_{\partial W}\otimes E_{\field{C}}^{\prime}$. We say that the data $\Xi_0:=(D_{\mathcal{E}},D_{\mathcal{E}'},D_{E},D_{E'})$ is a choice of Dirac operators for $\xi$. 

Whenever the choice of Dirac operators $\Xi_0$ is constructed from a Clifford connection $\nabla_W$ on $S_W$, $C^*_{\bf full}(\Gamma)$-connections $\nabla_{\mathcal{E}}$ and $\nabla_{\mathcal{E}'}$ on $\mathcal{E}_{C^*_{\bf full}(\Gamma)}$ respectively $\mathcal{E}_{C^*_{\bf full}(\Gamma)}'$, connections $\nabla_E$ and $\nabla_{E'}$ on $E_\C$ respectively $E_\C'$, we say that $\Xi_0$ is constructed from the choice of connection $(\nabla_W, \nabla_{\mathcal{E}},\nabla_{\mathcal{E}'},\nabla_{E},\nabla_{E'})$.
\end{define}

Our next goal is the construction of a map $\Phi_\mathcal{A}:\mathcal{S}^{geo}_0(\Gamma,\mathcal{A})\to \mathcal{S}^{h}_1(\sigma_1,\sigma_2)$ which combined with the relative $\eta$-invariant will give us an invariant for cycles in $\mathcal{S}^{geo}_0(\Gamma,\mathcal{A})$. 

\begin{define}[The map $\Phi_\mathcal{A}:\mathcal{S}^{geo}_0(\Gamma,\mathcal{A})\to\mathcal{S}^h_1(\sigma_1,\sigma_2)$]
\label{defOfPhi}
Assume that we have a cycle $(W,\xi,f)$ for $\mathcal{S}^{geo}_0(\Gamma,\mathcal{A})$ and a choice of Dirac operators $\Xi_0=(D_{\mathcal{E}},D_{\mathcal{E}'},D_{E},D_{E'})$ for $\xi=(\mathcal{E}_{\mathcal{A}(\Gamma)},\mathcal{E}^{\prime}_{\mathcal{A}(\Gamma)} ,E_{\field{C}}, E_{\field{C}}^{\prime} ,\alpha )$. We define the vector bundles 
$$G_i:= \mathcal{E}_{\mathcal{A}(\Gamma)}\otimes_{\sigma_i}\C^k \quad\mbox{and}\quad G_i':=\mathcal{E}'_{\mathcal{A}(\Gamma)}\otimes_{\sigma_i}\C^k.$$
We also set $G_{i,\C\ell}:=S_W\otimes G_i$ and $G_{i,\C\ell}':=S_W\otimes G_i$ and equip these Clifford bundles with the Dirac operators $D_{G_i}^W$ and $D_{G_i'}^W$, induced from $D_{\mathcal{E}}$ and $D_{\mathcal{E}'}$  respectively.

The Dirac operators $D_{E}$ and $D_{E'}$ induce Dirac operators $D_{E}^{\partial W\times [0,1]}$ and $D_{E'}^{\partial W\times [0,1]}$ on (respectively) the bundles 
\begin{align*}
E_{\C}\otimes S_{\partial W\times [0,1]}\to \partial W&\times [0,1]\quad\mbox{and}\\
&E'_{\C}\otimes S_{\partial W\times [0,1]}\to \partial W\times [0,1].
\end{align*}  
We also equip the Clifford bundles $(G_i|_{\partial W}\oplus E'_{\C}\otimes f^*E_i)\otimes S_{\partial W\times [0,1]}\to \partial W\times [0,1]$ with Dirac operators $\bar{D}_{i}^{\partial W\times [0,1]}$ such that 
\begin{align}
\nonumber
\bar{D}_{i}^{\partial W\times [0,1]}=D_{G_i}^W\oplus &D_{E',i}\quad\mbox{near} \quad 0\\
\label{barofnabla}
&\mbox{and}\quad \bar{D}_{i}^{\partial W\times [0,1]}=\alpha^*\left(D_{G_i'}^W\oplus D_{E,i}\right)\quad\mbox{near} \quad 1. 
\end{align}
Recalling the notation $d_i(D)$ from Equation \eqref{dinot}, we can define 
\begin{align*}
\Phi_{\mathcal{A}}(W,(\mathcal{E}_{\mathcal{A}(\Gamma)},&\mathcal{E}^{\prime}_{\mathcal{A}(\Gamma)} , E_{\field{C}}, E_{\field{C}}^{\prime} ,\alpha ),f):=\\
&(\partial W,S_{\partial W}\otimes E_{\C},f,D_{E},n)\dot{\cup}-(\partial W,S_{\partial W}\otimes E_{\C}',f,D_{E'},0),
\end{align*}
where
\begin{align*}
n:=\ind_{APS}(D_{G_1}^W)-\ind_{APS}(D_{G_2}^W)+&\\
\ind_{APS}(D_{G_1'}^{-W})-&\ind_{APS}(D_{G_2'}^{-W})+\\
&\ind_{APS}(\bar{D}_{1}^{\partial W\times [0,1]})-\ind_{APS}(\bar{D}_{2}^{\partial W\times [0,1]})+\\
+\dim \ker D^{\partial W}_{G_1}-\dim \ker D^{\partial W}_{G_2}&+\dim \ker D^{\partial W}_{G_1'}-\dim \ker D^{\partial W}_{G_2'}+\\
+d_1( D_{E}^{\partial W})-d_2(D_{E}^{\partial W})&+d_1( D_{E'}^{\partial W})-d_2(D_{E'}^{\partial W}).
\end{align*}
\end{define}

\begin{remark}
\label{ncomputation}
The expression for $n$ from Definition \ref{defOfPhi} can be rewritten in a form that is easier to work with using the Atiyah-Patodi-Singer index theorem. Let $\Xi_0$ be a choice of Dirac operators constructed (as in the previous paragraphs) from the choice of connection $(\nabla_W, \nabla_{\mathcal{E}},\nabla_{\mathcal{E}'},\nabla_{E},\nabla_{E'})$. Let $\nabla_{G_i}:=(\sigma_i)_*\nabla_{\mathcal{E}}$ and $\nabla_{G_i'}:=(\sigma_i)_*\nabla_{\mathcal{E}'}$. We also let $\bar{\nabla}_i:=(1-t)\nabla_{G_i}\oplus \nabla_{E'}+t\alpha^*(\nabla_{G_i'}\oplus \nabla_{E})$, i.e. the connection on $(G_i\oplus E'_{\C}\otimes f^*E_i)\times [0,1]\to \partial W\times [0,1]$ that with the connection on $\partial W\times [0,1]$ induced from $\nabla_W$ induces the Dirac operator $\bar{D}_i$. The Atiyah-Patodi-Singer index theorem (see Equation \eqref{APSformula} on page \pageref{APSformula}) implies that the integer $n$ takes the form
\begin{align*}
&n=\int_W \left(\ch(\nabla_{G_1})-\ch(\nabla_{G_2})\right)\wedge Td(\nabla_W)-\int_W \left(\ch(\nabla_{G'_1})-\ch(\nabla_{G'_2})\right)\wedge Td(\nabla_W)+\\
&\qquad\qquad\qquad\quad+\int_{\partial W\times [0,1]} \left(\ch(\bar{\nabla}_1)-\ch(\bar{\nabla}_2)\right)Td(\nabla_{\partial W})+\\
&\qquad\qquad+\frac{1}{2}\left(\eta(D_{E',1}^{\partial W})-\eta(D_{E',2}^{\partial W})\right)-\frac{1}{2}\left(\eta(D_{E,1}^{\partial W})-\eta(D_{E,2}^{\partial W})\right)\\
&\qquad\qquad+\frac{1}{2}\left(d_1( D_{E'}^{\partial W})-d_2( D_{E'}^{\partial W})\right)-\frac{1}{2}\left(d_1(D_{E}^{\partial W})-d_2( D_{E}^{\partial W})\right).
\end{align*}
This expression is the motivation for ``correcting" by the dimensions of kernels on the boundary; it forces the contributions from the kernels on the boundary to respect orientation in the same way the $\eta$-terms do. In particular, if $\mathcal{E}_{\mathcal{A}(\Gamma)}'=E_{\C}'=0$ then
\[n=\ind_{APS}(\hat{D}^W_1)-\ind_{APS}(\hat{D}^W_2),\]
where $\hat{D}_i^W$ is a Dirac operator that equals $\alpha^*(D_{E,i}^{\partial W\times [0,1]})$ in a collar neighborhood of $\partial W$ and equals $D_{G_i}$ inside $W$. The reader may find it useful to compare the definition of $\Phi_\mathcal{A}$ given here to expression \eqref{phioneasy} (on page \pageref{phioneasy}), which deals with the case of easy cycles.
\end{remark}

\subsection{Well defined from cycles to classes}

Our first goal is to prove that $\Phi_\mathcal{A}$ is well-defined; firstly, as a map from cycles to classes in $\mathcal{S}^h_1(\sigma_1, \sigma_2)$ and then as a map $\mathcal{S}^{geo}_0(\Gamma,\mathcal{A})\to\mathcal{S}^h_1(\sigma_1,\sigma_2)$. To do so, we require a number of lemmas.

\begin{lemma}
\label{cyclepart}
Let $M$ be a closed orientable manifold, $S\to M$ a Clifford bundle, $f:M \rightarrow B\Gamma$ a continuous map. Let $\nabla$ and $\nabla^{\prime}$ be two Clifford connections on $S$ leading to Dirac operators $D$ and $D^{\prime}$ respectively. We can form 
\begin{enumerate}
\item The manifold with boundary $M\times [0,1]$ (along with the pullback of the orientation on $M$ and the pullback of $S$). 
\item A Clifford connection $\bar{\nabla}=(1-t)\nabla+t\nabla'$ and the associated Dirac operator $D_{\bar{\nabla}}$.
\end{enumerate}
Let $D_{\bar{\nabla},i}$ be $D_{\bar{\nabla}}$ twisted by the flat connection on $f^*E_i$ and set
$$m:= \ind_{APS} (D_{\bar{\nabla},1})-\ind_{APS}(D_{\bar{\nabla},2}).$$ 
Then the cycle 
$$((M,S_{\C\ell},f)\dot{\cup} (M,-S_{\C\ell},f), D\dot{\cup} -D^{\prime},m)$$
is a boundary and therefore trivial in $\mathcal{S}^h_*(\sigma_1, \sigma_2)$.
\end{lemma}

\begin{proof}
The result follows from the definition of bordism in $\mathcal{S}^h(\sigma_1, \sigma_2)$. The construction of the bordism we have in mind uses the manifold with boundary $M\times [0,1]$, the pullback of the relevant bundle, and the continuous function $f\circ \pi$ (here $\pi:M\times [0,1]\rightarrow M$ denotes the projection map). We leave the details to the reader.
\end{proof}

\begin{lemma}
\label{identieta}
Let $M$ be a closed oriented manifold. Suppose that $S_1$ and $S_2$ are Clifford bundles over $M$ with fixed Clifford connections; we denote the connections by $\nabla_{S_i}$. Let $D_{S_i}$ ($i=1$ and $2$) and $D_{S_1\oplus S_2}$ denote the associated Dirac operators; for the operator $D_{S_1\oplus S_2}$ we used the natural Clifford connection coming from the ones on $S_1$ and $S_2$. Then $ \eta(D_{S_1 \oplus S_2}) = \eta(D_{S_1}) + \eta(D_{S_2}).$

Furthermore, let $V_1$ and $V_2$ also denote Clifford bundles over $M$ (again with fixed Clifford connections, which are denoted by $\nabla_{V_i}$). We form the Dirac operators $D_{S_1 \oplus V_2}$ and $D_{S_2 \oplus V_1}$. Suppose that there exists an isomorphism $\alpha: S_1 \oplus V_2 \cong S_2\oplus V_1$ of Clifford bundles and that the Clifford connections are compatible with this isomorphism; that is
$$\alpha^*(\nabla_{S_2}\otimes I + I\otimes \nabla_{V_1})=\nabla_{S_1}\otimes I + I \otimes \nabla_{V_2}$$
Then $\eta(D_{S_1\oplus V_2}) = \eta(D_{S_2\oplus V_1}).$
\end{lemma}

\begin{proof}
Both the statements follow by definition.
\end{proof}

\begin{lemma}
\label{thisCycleIsABoundary}
Let $W$ be a compact spin$^c$-manifold with boundary, $E$ a vector bundle over it, $S_W\to W$ its spin$^c$-structure and $\nabla_E$ and $\nabla^{\prime}_E$ are two connections $E$, $\nabla_W$ and $\nabla_W'$ are two Clifford connections on $S_W$, and all of the connections $\nabla_E$, $\nabla^{\prime}_E$, $\nabla_W$ and $\nabla_W'$ are of product type near $\partial W$. Then 
\small
\begin{align*}
\int_W\ch(\nabla'_E)\wedge Td(\nabla_W') + &\int_{\partial W\times [0,1]}\ch(\nabla_E,\nabla_E')\wedge Td(\bar{\nabla}_{\partial W\times [0,1]}) + \int_{-W} \ch(\nabla_E)\wedge Td(\nabla_W)=0.
\end{align*}
\normalsize
Here, $t$ denotes the coordinate in $[0,1]$ and the required data is constructed as follows:
\begin{enumerate}
\item $\ch(\nabla_E,\nabla_E')$ denotes the Chern character of $\bar{\nabla}_{E}:=t\nabla_E+(1-t)\nabla_E'$;
\item  $\bar{\nabla}_{\partial W\times [0,1]}:=t\nabla_{\partial W}+(1-t)\nabla_{\partial W}'$, where $\nabla_{\partial W}$ and $\nabla_{\partial W}'$ denote the Clifford connections on $S_W|_{\partial W}\times [0,1]\to \partial W\times [0,1]$ induced from $\nabla_W$ respectively $\nabla_W'$.
\end{enumerate}
\end{lemma}

\begin{proof}
Define $Z:=W\cup_{\partial W} \partial W \times [0,1]\cup_{\partial W} -W$ and define $\tilde{S}\to Z$ by gluing together $S_W\otimes E\to W$ over the cylinder with $-S_W\otimes E\to -W$. Let $D_{\tilde{S}}$ denote the Dirac operator on $\tilde{S}$ constructed from gluing together the Clifford connection $\nabla_W\otimes \nabla_E$ on $W$ with $\nabla'_W\otimes \nabla_E'$ over $-W$ over the cylinder using the Clifford connection $\bar{\nabla}_{\partial W\times [0,1]}\otimes \bar{\nabla}_E$. By the local index theorem, 
\begin{align*}
\ind_{AS}(D_{\tilde{S}}) =& \int_W\ch(\nabla'_E)\wedge Td(\nabla_W') + \\
&\int_{\partial W\times [0,1]}\ch(\bar{\nabla}_E)\wedge Td(\bar{\nabla}_{\partial W\times [0,1]}) + \int_{-W} \ch(\nabla_E)\wedge Td(\nabla_W).
\end{align*}
We note that $W\times [0,1]$ can be made into a smooth manifold with boundary (using the straightening of the angle technique, see Lemma $4.1.9$ of \cite{Rav}) and 
$$\partial(W\times [0,1])=Z \quad\mbox{with}\quad S_W\otimes E\times [0,1]|_Z=\tilde{S}.$$ 
That is, $(Z,\tilde{S})=\partial (W\times [0,1],S_W\otimes E\times [0,1])$ so $\ind_{AS} (D_{\tilde{S}}) = 0$ since the index is a bordism invariant.
\end{proof}

\begin{lemma}
The map, $\Phi_\mathcal{A}$, is well-defined as a map from cycles with vector bundle data in $\mathcal{S}^{geo}_0(\Gamma,\mathcal{A})$ to classes in $\mathcal{S}^h_1(\sigma_1,\sigma_2)$.
\end{lemma}

\begin{proof}
Let $x=(W,(\mathcal{E}_{\mathcal{A}(\Gamma)},\mathcal{E}^{\prime}_{\mathcal{A}(\Gamma)} , E_{\field{C}}, E_{\field{C}}^{\prime} ,\alpha ),f)$ be a cycle with vector bundle data in $\mathcal{S}^{geo}_0(\Gamma,\mathcal{A})$. We must show that $\Phi_\mathcal{A}(x)$ does not depend on the particular choice of Dirac operators. As such, let $\Xi_0=(D_{\mathcal{E}},D_{\mathcal{E}'},D_{E},D_{E'})$ and $\tilde{\Xi}_0=(\tilde{D}_{\mathcal{E}},\tilde{D}_{\mathcal{E}'},\tilde{D}_{E},\tilde{D}_{E'})$ denote two choices of Dirac operators constructed from two choices of connections 
$$(\nabla_W, \nabla_{\mathcal{E}},\nabla_{\mathcal{E}'},\nabla_{E},\nabla_{E'})\quad\mbox{respectively}\quad (\tilde{\nabla}_W, \tilde{\nabla}_{\mathcal{E}},\tilde{\nabla}_{\mathcal{E}'},\tilde{\nabla}_{E},\tilde{\nabla}_{E'}).$$ 
With these choices made, we have two associated classes $\Phi_{\Xi_0}$ and $\Phi_{\tilde{\Xi}_0}$ in the group $\mathcal{S}^h_*(\sigma_1,\sigma_2)$. Lemma \ref{cyclepart} implies that
$$\Phi_{\Xi_0}-\Phi_{\tilde{\Xi}_0}=(\emptyset,\emptyset,\emptyset,\emptyset,N),$$
where 
\[N=n-\tilde{n}-m-m'+d_2(D^{\partial W}_{E'})-d_1( D^{\partial W}_{E'})+d_2( \tilde{D}^{\partial W}_{E})-d_1( \tilde{D}^{\partial W}_{E}).\]
Here $n$ and $\tilde{n}$ are as in Definition \ref{defOfPhi}. The number $m$ is the difference of Atiyah-Patodi-Singer indices of the Dirac operator associated with interpolating between the Dirac operators $D^{\partial W}_{E,i}$ and  $\tilde{D}^{\partial W}_{E,i}$ on the cylinder $\partial W\times [0,1]$. Analogously, $m'$ is the difference of Atiyah-Patodi-Singer indices but for $\tilde{D}^{\partial W}_{E',i}$ and  $D^{\partial W}_{E',i}$. Note that $m'$ can also be defined as the difference of Atiyah-Patodi-Singer indices of the Dirac operator associated with interpolating between the Dirac operators $D^{\partial W}_{E',i}$ and  $\tilde{D}^{\partial W}_{E',i}$ when equipping the cylinder $\partial W\times [0,1]$ with its opposite orientation.

It follows from the Atiyah-Patodi-Singer theorem that 
$$N=L+S+D,$$
where 
\begin{enumerate}
\item $L$ is computed via local expressions;
\item $S$ is computed from the spectral data of the Dirac operators on the boundaries;
\item $D$ is computed from the dimensions of the kernels of the Dirac operators on the boundaries. 
\end{enumerate}
We follow the notations of Remark \ref{ncomputation} and the statement of Lemma \ref{thisCycleIsABoundary} with the addendum that terms computed using $\tilde{\Xi}_0$ will be indicated with a tilde. We begin with the local term, which (after rearranging the terms and applying Lemma \ref{thisCycleIsABoundary}) is given by
\small
\begin{align*}
L=&\int_W(\ch(\nabla_{G_1})-\ch(\nabla_{G_2}))\wedge Td(\nabla_W)-\int_W(\ch(\nabla_{G_1'})-\ch(\nabla_{G_2'}))\wedge Td(\nabla_W)-\\
&-\int_{\partial W\times [0,1]} \ch(\nabla_{G_1}\oplus \nabla_{E'\otimes E_1},\alpha^*(\nabla_{G_1'}\oplus \nabla_{E\otimes E_1}))\wedge Td(\nabla_{\partial W})+\\
&+\int_{\partial W\times [0,1]} \ch(\nabla_{G_2}\oplus \nabla_{E'\otimes E_2},\alpha^*(\nabla_{G_2'}\oplus \nabla_{E\otimes E_2}))\wedge Td(\nabla_{\partial W})-\\
&-\int_W(\ch(\tilde{\nabla}_{G_1})-\ch(\tilde{\nabla}_{G_2}))\wedge Td(\tilde{\nabla}_W)+\int_W(\ch(\tilde{\nabla}_{G_1'})-\ch(\tilde{\nabla}_{G_2'}))\wedge Td(\tilde{\nabla}_W)+\\
&+\int_{\partial W\times [0,1]} \ch(\tilde{\nabla}_{G_1}\oplus \tilde{\nabla}_{E'\otimes E_1},\alpha^*(\tilde{\nabla}_{G_1'}\oplus \tilde{\nabla}_{E\otimes E_1}))\wedge Td(\tilde{\nabla}_{\partial W})-\\
&-\int_{\partial W\times [0,1]} \ch(\tilde{\nabla}_{G_2}\oplus \tilde{\nabla}_{E'\otimes E_2},\alpha^*(\tilde{\nabla}_{G_2'}\oplus \tilde{\nabla}_{E\otimes E_2}))\wedge Td(\tilde{\nabla}_{\partial W})-\\
&-\int_{\partial W\times [0,1]} (\ch(\nabla_{E\otimes E_1}, \tilde{\nabla}_{E\otimes E_1})- \ch(\nabla_{E\otimes E_2}, \tilde{\nabla}_{E\otimes E_2}))\wedge Td(\bar{\nabla}_{\partial W\times [0,1]})-\\
&-\int_{\partial W\times [0,1]} \big(\ch(\nabla_{E'\otimes E_1}, \tilde{\nabla}_{E'\otimes E_1})- \ch(\nabla_{E'\otimes E_2}, \tilde{\nabla}_{E'\otimes E_2})\big)\wedge Td(\bar{\nabla}_{\partial W\times [0,1]})=0
\end{align*}
\normalsize 
The Atiyah-Patodi-Singer index theorem and Remark \ref{ncomputation} imply that the spectral term is given by
\small
\begin{align*}
2S=&\eta(D^{\partial W}_{E'\otimes E_1})-\eta(D^{\partial W}_{E'\otimes E_2})-\eta(D^{\partial W}_{E\otimes E_1})+\eta(D^{\partial W}_{E\otimes E_2})\\
&-\eta(\tilde{D}^{\partial W}_{E'\otimes E_1})+\eta(\tilde{D}^{\partial W}_{E'\otimes E_2})+\eta(\tilde{D}^{\partial W}_{E\otimes E_1})-\eta(\tilde{D}^{\partial W}_{E\otimes E_2})\\
&-\eta(\tilde{D}^{\partial W}_{E\otimes E_1})+\eta(\tilde{D}^{\partial W}_{E\otimes E_2})+\eta(D^{\partial W}_{E\otimes E_1})-\eta(D^{\partial W}_{E\otimes E_2})\\
&+\eta(\tilde{D}^{\partial W}_{E'\otimes E_1})-\eta(\tilde{D}^{\partial W}_{E'\otimes E_2})-\eta(D^{\partial W}_{E\otimes E_1})+\eta(D^{\partial W}_{E'\otimes E_2})=0.\\
\end{align*}
\normalsize
Here, the first four terms come from $n$, the next four from $\tilde{n}$ and the remaining eight come from $m+m'$. Finally, letting $n_D$, $\tilde{n}_D$, $m_D$ and $m'_D$ denote the dimensional contributions to $n$, $\tilde{n}$, $m$ respectively $m'$ we have that
\scriptsize
\begin{align*}
2D=&2n_D-2\tilde{n}_D-2m_D-2m'_D+2d_2( D^{\partial W}_{E'})-2d_1( D^{\partial W}_{E'})+2d_2(\tilde{D}^{\partial W}_{E})-2d_1( \tilde{D}^{\partial W}_{E})=\\
=&d_1( D_{E'}^{\partial W})-d_2( D_{E'}^{\partial W})-d_1( D_{E}^{\partial W})+d_2( D_{E}^{\partial W})\\
&\qquad\qquad\qquad\qquad-d_1( \tilde{D}_{E'}^{\partial W})+d_2( \tilde{D}_{E'}^{\partial W})+d_1( \tilde{D}_{E}^{\partial W})-d_2( \tilde{D}_{E}^{\partial W})\\
&+d_1( D_{E}^{\partial W})-d_2( D_{E}^{\partial W})+d_1( \tilde{D}_{E_{\C}}^{\partial W})-d_2( \tilde{D}_{E_{\C}}^{\partial W})\\
&+d_1( D_{E'}^{\partial W})-d_2( D_{E'}^{\partial W})+d_1( \tilde{D}_{E'}^{\partial W})-d_2( \tilde{D}_{E'}^{\partial W})\\
&\qquad\qquad\qquad\qquad+2d_2( D^{\partial W}_{E'})-2d_1( D^{\partial W}_{E'})+2d_2( \tilde{D}^{\partial W}_{E})-2d_1( \tilde{D}^{\partial W}_{E})=\\
=&d_1( D_{E'}^{\partial W})-d_2( D_{E'}^{\partial W})-d_1( D_{E}^{\partial W})+d_2( D_{E}^{\partial W})\\
&\qquad\qquad\qquad\qquad-d_1( \tilde{D}_{E'}^{\partial W})+d_2( \tilde{D}_{E'}^{\partial W})+d_1( \tilde{D}_{E}^{\partial W})-d_2( \tilde{D}_{E}^{\partial W})\\
&+d_1( D_{E}^{\partial W})-d_2( D_{E}^{\partial W})-d_1( \tilde{D}_{E_{\C}}^{\partial W})+d_2( \tilde{D}_{E_{\C}}^{\partial W})\\
&\qquad\qquad\qquad\qquad-d_1( D_{E'}^{\partial W})+d_2( D_{E'}^{\partial W})+d_1( \tilde{D}_{E'}^{\partial W})-d_2(\tilde{D}_{E'}^{\partial W})=0.
\end{align*}
\normalsize
This proves that $N=0$; hence, $\Phi_\mathcal{A}(x)$ does not depend on the choice of Dirac operators. 
\end{proof}

\subsection{$\Phi_\mathcal{A}$ is well defined from classes to classes}

In this section, we prove that $\Phi_\mathcal{A}$ induces a well defined map $\mathcal{S}_0^{geo}(\Gamma;\mathcal{A})\to \mathcal{S}_1^h(\sigma_1,\sigma_2)$.

\begin{lemma}
\label{bordantcyclepart}
Assume that $(M,S_{\C\ell},f,D^M,m)=\partial(W,S^W_{\C\ell},g,Q)$, in the sense of Definition \ref{apicyclewithboundary}, then for any $n\in \Z$ there is a bordism
\[(M,S_{\C\ell},f,D^M,n)\sim_{bor}(\emptyset,\emptyset,\emptyset,\emptyset, n-m),\]
where (the above bordism condition implies that)
\[m=\ind_{APS}(Q_1)-\ind_{APS}(Q_2).\]
\end{lemma}
\begin{proof}
It follows directly from the definition of bordism that $(M,S_{\C\ell},f,D^M,m)\sim_{bor}(\emptyset,\emptyset,\emptyset,\emptyset,0)$. The Lemma follows from the fact that the bordism relation respects addition of cycles.
\end{proof}

\begin{lemma}
\label{trivialdata}
If the cycle $x=(W,(\mathcal{E}_{\mathcal{A}(\Gamma)},\mathcal{E}^{\prime}_{\mathcal{A}(\Gamma)} , E_{\C}, E_{\C}^{\prime} ,\alpha ),f)$ is such that $f$ extends to a function $g:W\to B\Gamma$ and there are bundles $F,F'\to W$ extending $E$ respectively $E'$ and $\alpha$ extends to an isomorphism on $W$
\[\beta:\mathcal{E}_{\mathcal{A}(\Gamma)}\oplus F'\otimes g^*\mathcal{L}\cong \mathcal{E}'_{\mathcal{A}(\Gamma)}\oplus F\otimes g^*\mathcal{L},\]
then $\Phi_\mathcal{A}(x)=(\emptyset,\emptyset,\emptyset,\emptyset,0)$.
\end{lemma}

\begin{proof}
We choose Dirac operators $D_{F}^W$ and $D_{F'}^W$ on $S_W\otimes F$ respectively $S_W\otimes F'$ of product type as in Equation \eqref{diracofproducttupe} (found on page \pageref{diracofproducttupe}) near $\partial W$ with boundary operators $D_{E}^{\partial W}$ respectively $D_{E'}^{\partial W}$. By Lemma \ref{bordantcyclepart}, 
\[\Phi_\mathcal{A}(x)=(\emptyset,\emptyset,\emptyset,\emptyset,n-m-m'),\]
where $n$ is as in Definition \ref{defOfPhi}\footnote{Compare with Remark \ref{ncomputation}.}, 
\begin{align*}
m:=\ind_{APS}(D^W_{F,1})-&\ind_{APS}(D^W_{F,2})\quad\mbox{and}\\ 
&m':=\ind_{APS}(D_{F',1}^{-W})-\ind_{APS}(D_{F',2}^{-W}).
\end{align*}
We define $Z:=W\cup_{\partial W} \partial W\times [0,1]\cup_{\partial W} -W$ and, for $i=1,2$, the Clifford bundles on $Z$;
\small
\[\mathcal{G}_i:=(G_i \oplus F'\otimes E_i)\otimes S_W\cup_{\partial W\times \{0\}} \big((G_i\oplus  E'\otimes E_i)\otimes S_{\partial W\times [0,1]}\big)\cup_{\partial W\times \{1\}}(G_i'\oplus F\otimes E_i)\otimes S_{-W}.\]
\normalsize
We can equip $\mathcal{G}_i$ with a Dirac operator by gluing together $D_{G_i}^W\oplus D_{F'\otimes E_i}^W$ with $\bar{D}_i$ over $\partial W\times \{0\}$ and with $\beta^*(D_{G_i'}^{-W}\oplus D_{F \otimes E_i}^{-W})$ over $\partial W\times \{1\}$.

It follows from the Atiyah-Patodi-Singer index theorem that
\small
\begin{align*}
n-m-&m'=\\
&=\int_W \left(\ch(\nabla_{G_1})-\ch(\nabla_{G_2})\right)\wedge Td(\nabla_W)-\int_W \left(\ch(\nabla_{G'_1})-\ch(\nabla_{G'_2})\right)\wedge Td(\nabla_W)+\\
&\qquad\qquad+\int_{\partial W\times [0,1]} \left(\ch(\bar{\nabla}_1)-\ch(\bar{\nabla}_2)\right)Td(\nabla_{\partial W})-\\
&\qquad\qquad\qquad\qquad-\int_W \left(\ch(\nabla_{F\otimes E_1})-\ch(\nabla_{F\otimes E_2})\right)\wedge Td( \nabla_W)+\\
&\qquad\qquad\qquad\qquad\qquad\qquad+\int_W \left(\ch(\nabla_{F' \otimes E_1})-\ch(\nabla_{F'\otimes E_2})\right)\wedge Td(\nabla_W)=\\
&=\int_W \left(\ch(\nabla_{G_1}\oplus \nabla_{F' \otimes E_1})-\ch(\nabla_{G_2}\oplus \nabla_{F' \otimes E_2})\right)\wedge Td(\nabla_W)-\\
&\qquad\qquad-\int_W \left(\ch(\nabla_{G_1'}\oplus \nabla_{F \otimes E_1})-\ch(\nabla_{G'_2}\oplus \nabla_{F \otimes E_2})\right)\wedge Td(\nabla_W)+\\
&\qquad\qquad\qquad\qquad+\int_{\partial W\times [0,1]} \left(\ch(\bar{\nabla}_1)-\ch(\bar{\nabla}_2)\right)Td(\nabla_{\partial W})=\\
&=\int_W \left(\ch(\nabla_{G_1}\oplus \nabla_{F' \otimes E_1})-\ch(\nabla_{G_2}\oplus \nabla_{F' \otimes E_2})\right)\wedge Td(\nabla_W)-\\
&\qquad\qquad-\int_W \left(\ch\beta^*(\nabla_{G_1'}\oplus \nabla_{F \otimes E_1})-\ch\beta^*(\nabla_{G'_2}\oplus \nabla_{F\otimes E_2})\right)\wedge Td(\nabla_W)+\\
&\qquad\qquad\qquad\qquad+\int_{\partial W\times [0,1]} \left(\ch(\bar{\nabla}_1)-\ch(\bar{\nabla}_2)\right)Td(\nabla_{\partial W})=\\
&=\ind(Z,\mathcal{G}_1)-\ind(Z,\mathcal{G}_2)=0,
\end{align*}
\normalsize
The reader should note that the last equality follows from the bordism invariance of the index on closed manifolds.
\end{proof}

\begin{lemma}
\label{gluingcyclesalongboundary}
The map $\Phi_\mathcal{A}$ respects gluing of cycles in the following sense. Assume that $x_i=(W_i,(\mathcal{E}_{\mathcal{A}(\Gamma),i},\mathcal{E}^{\prime}_{\mathcal{A}(\Gamma),i} , E_{\field{C},i}, E_{\field{C},i}^{\prime} ,\alpha_i ),f_i)$ are cycles for $i=1,2$ and there is an spin$^c$ manifold $Y$ satisfying that
$$\partial W_i=M_i\dot{\cup}((-1)^i Y),\quad f_1|_Y=f_2|_Y\quad\mbox{and}$$
$$(\mathcal{E}_{\mathcal{A}(\Gamma),1},\mathcal{E}^{\prime}_{\mathcal{A}(\Gamma),1} , E_{\field{C},1}, E_{\field{C},1}^{\prime} ,\alpha_1 )|_Y=(\mathcal{E}_{\mathcal{A}(\Gamma),2},\mathcal{E}^{\prime}_{\mathcal{A}(\Gamma),2} , E_{\field{C},2}, E_{\field{C},2}^{\prime} ,\alpha_2)|_Y.$$
Then, we can form the cycle 
$$x_1\cup_Y x_2:=(W,(\mathcal{E}_{\mathcal{A}(\Gamma)},\mathcal{E}^{\prime}_{\mathcal{A}(\Gamma)} , E_{\field{C}}, E_{\field{C}}^{\prime} ,\alpha ),f),$$
where $W:=W_1\cup_Y W_2$ and the vector bundle data is defined as
\begin{align*}
\mathcal{E}_{\mathcal{A}(\Gamma)}&:=\mathcal{E}_{\mathcal{A}(\Gamma),1}\cup_Y\mathcal{E}_{\mathcal{A}(\Gamma),2},\quad \mathcal{E}_{\mathcal{A}(\Gamma)}':=\mathcal{E}_{\mathcal{A}(\Gamma),1}'\cup_Y\mathcal{E}_{\mathcal{A}(\Gamma),2}' ,\\ 
E_{\field{C}}&:=E_{\field{C},1}|_{M_1}\dot{\cup} E_{\field{C},2}|_{M_2},\quad E_{\field{C}}':=E'_{\field{C},1}|_{M_1}\dot{\cup} E'_{\field{C},2}|_{M_2},\\
\alpha&:=\alpha_1|_{M_1}\dot{\cup} \alpha_2|_{M_2}\quad\mbox{and}\quad  f:=f_1|_{M_1}\dot{\cup} f_2|_{M_2}.
\end{align*}
Moreover, we have that 
$$\Phi_\mathcal{A}(x_1\dot{\cup}x_2)=\Phi_\mathcal{A}(x_1\cup_Y x_2)\quad\mbox{in}\quad \mathcal{S}^h_1(\sigma_1,\sigma_2)$$
\end{lemma}

\begin{proof}
We choose connections compatible with the gluing near $Y$. Let $n_1$ and $n_2$ be the integer parts of $\Phi_\mathcal{A}(x_1)$ respectively $\Phi_\mathcal{A}(x_2)$. Set $g:=f_1|Y=f_2|_Y$, $F_{\field{C}}:=E_{\field{C},1}|_Y=E_{\field{C},1}|_Y$ and $F'_{\field{C}}:=E_{\field{C},1}'|_Y=E_{\field{C},2}'|_Y$. Let $S_Y\to Y$ denote the spin$^c$-structure on $Y$, similarly we let $S_{\partial W_1}$, $S_{\partial W_2}$ and $S_{\partial W}$ denote the indicated spin$^c$-structure. It follows that
\begin{align*}
\Phi_\mathcal{A}&(x_1\dot{\cup}x_2)=\\
=&(\partial W_1\dot{\cup} \partial W_2, S_{\partial W_1}\otimes E_{\C,1}\dot{\cup} S_{\partial W_2}\otimes E_{\C,2}, f_1\dot{\cup} f_2, D^{\partial W_1}_{E_{1}}\dot{\cup} D^{\partial W_2}_{E_{2}},n_1+n_2)\dot{\cup}\\
&\dot{\cup}-(\partial W_1\dot{\cup} \partial W_2,S_{\partial W_1}\otimes E'_{\C,1}\dot{\cup} S_{\partial W_2}\otimes E'_{\C,2}, f_1\dot{\cup} f_2, D^{\partial W_1}_{E_{1}'}\dot{\cup} D^{\partial W_2}_{E_{2}'},0)=\\
=&(\partial W,S_{\partial W}\otimes E_\C,f,D^{\partial W}_{E},n_1+n_2)\dot{\cup}-(\partial W,S_{\partial W}\otimes E_\C',f,D^{\partial W}_{E'},0)\dot{\cup}\\
&\quad\dot{\cup}(-Y\dot{\cup} Y, -S_{Y}\otimes F_\C\dot{\cup} S_Y\otimes F_{\C}, g\dot{\cup} g, -D^{Y}_{F}\dot{\cup} D^{Y}_{F},0)\\
&\qquad\dot{\cup}(Y\dot{\cup}-Y, S_Y\otimes F'_{\C}\dot{\cup} -S_Y\otimes F'_{\C}, g\dot{\cup} g, D^{Y}_{F'}\dot{\cup} -D^{Y}_{F'},0).
\end{align*}
By Lemma \ref{bordantcyclepart}, 
\begin{align*}
(-Y\dot{\cup} Y, -S_{Y}\otimes F_\C\dot{\cup} S_Y\otimes F_{\C}, &g\dot{\cup} g, -D^{Y}_{F}\dot{\cup} D^{Y}_{F},0)\\
\dot{\cup}(Y\dot{\cup}-Y, S_Y\otimes F'_{\C}&\dot{\cup} -S_Y\otimes F'_{\C}, g\dot{\cup} g, D^{Y}_{F'}\dot{\cup} -D^{Y}_{F'},0)\sim_{bor} \\
&\qquad\qquad\qquad\qquad\qquad\qquad\sim_{bor}(\emptyset,\emptyset,\emptyset,\emptyset,-m),
\end{align*}
where 
\begin{align*}
m:=&\ind_{APS}(D_{F\otimes E_1}^{-Y\times [0,1]})-\ind_{APS}(D_{F\otimes E_2}^{-Y\times [0,1]})-\\
&-\ind_{APS}(D^{Y\times [0,1]}_{F'\otimes E_1})+\ind_{APS}(D^{Y\times [0,1]}_{F'\otimes E_2})=\\
=&d_1(D^Y_{F})-d_2(D^Y_F)-d_1(D^Y_{F'})+d_2(D^Y_{F'}).
\end{align*}
Let $n$ denote the integer part of $\Phi_\mathcal{A}(x_1\cup_Y x_2)$. The proof of the Lemma is complete upon noting that $n=n_1+n_2-m$; this equality follows from the Atiyah-Patodi-Singer index theorem and Remark \ref{ncomputation}. 
\end{proof}
\begin{remark}
The reader should compare the proof of Lemma \ref{gluingcyclesalongboundary} to Proposition \ref{gluap}. Moreover, Lemma \ref{gluingcyclesalongboundary} proves a special case of bordism invariance for $\Phi_\mathcal{A}$. Namely, using the notation of Lemma \ref{gluingcyclesalongboundary}, we have that $x_1\dot{\cup}x_2\sim_{bor} x_1\cup_Y x_2$ in $\mathcal{S}^{geo}_0(\Gamma,\mathcal{A})$. A specific bordism can be constructed using the manifold with boundary obtained from straightening the angle on $Z=(W\cup_Y Y\times [0,1]\cup_Y W')\times [0,1]$ and considering the regular domain $(W\cup_Y Y\times [0,1]\cup_Y W')\times \{0\}\dot{\cup}(W\dot{\cup} W')\times \{1\}$ of $\partial Z$. 
\end{remark}

\begin{theorem}
\label{phiacommutingdiagram}
Assuming that $\mathcal{A}(\Gamma)$ is a Banach algebra closure of $\C[\Gamma]$ such that $\sigma_1$ and $\sigma_2$ extends to $\mathcal{A}(\Gamma)$, the map $\Phi_\mathcal{A}:\mathcal{S}^{geo}_0(\Gamma,\mathcal{A})\to\mathcal{S}^{h}_1(\sigma_1, \sigma_2)$ is well-defined and fits into a commutative diagram with exact rows:
\begin{equation}
\label{commutingsigmaoneandtwo}
\begin{CD}
@.K_0^{geo}(pt;\mathcal{A}(\Gamma))@>r>>\mathcal{S}^{geo}_0(\Gamma,\mathcal{A})@>\delta>>K_1^{geo}(B\Gamma) \\
@. @V(\sigma_1-\sigma_2)_*VV  @V\Phi_\mathcal{A} VV @VVV  @. \\
0@>>>\Z @>>> S^{h}_1(\sigma_1,\sigma_2)@>\delta_{\sigma_1,\sigma_2}>> K_1^h(B\Gamma)@>>>0,
\end{CD},
\end{equation}
where the lower row is the short exact sequence of Proposition \ref{sesfortherelativegroups} and the right vertical map is the isomorphism of Lemma \ref{khomorientedlem}.
\end{theorem}

\begin{proof}
Let us start by showing that $\Phi_\mathcal{A}$ is well defined from classes of cycles with vector bundle data (i.e., map degenerate cycles to $0$ and respects disjoint union/direct sum, vector bundle modification and bordism). It follows from Remark \ref{ncomputation} and Lemma \ref{cyclepart} that if $(W,\xi,f)$ is a degenerate cycle then $\Phi_\mathcal{A}(W,\xi,f)$ is null bordant. The result for the disjoint union/direct sum relation is clear. 

To prove invariance under vector bundle modification, consider a cycle 
$$(W,(\mathcal{E}_{\mathcal{A}(\Gamma)},\mathcal{E}^{\prime}_{\mathcal{A}(\Gamma)} , E_{\C}, E_{\C}^{\prime} ,\alpha ),f)$$ 
for $\mathcal{S}^{geo}_0(\Gamma,\mathcal{A})$ and $V\to W$ a spin$^c$ vector bundle with even-dimensional fibers. We take $n$ as in Definition \ref{defOfPhi}. By \cite[Proposition $6.6$]{HReta} the Atiyah-Patodi-Singer index is invariant under vector bundle modification when choosing the Dirac operators $D^{\partial W^V}$ on $\partial W^V$ in the canonical way (see \cite[Proposition 6.6]{HReta}). Let $\pi_V:(\partial W)^V\to \partial W$ denote the projection. It follows from Proposition \ref{tauvversusbott} and these observations that
\begin{align*}
\Phi_\mathcal{A}((W,&(\mathcal{E}_{\mathcal{A}(\Gamma)},\mathcal{E}^{\prime}_{\mathcal{A}(\Gamma)} , E_{\C}, E_{\C}^{\prime} ,\alpha ),f)^V)=\\
&=(\partial W^V,S_{\partial W^V}\otimes E^V_\C,f\circ \pi_V,D^{\partial W^V}_{E^V},n)\\
&\qquad \dot{\cup}-(\partial W^V,S_{\partial W^V}\otimes (E'_\C)^V,f\circ \pi_V,D^{\partial W^V}_{(E')^V},0)=\\
&=(\partial W,S_{\partial W}\otimes E_\C,f,D^{\partial W}_{E},n)^{V|_{\partial W}}\dot{\cup}-(\partial W,S_{\partial W}\otimes E_\C',f,D^{\partial W}_{E'},0)^{V|_{\partial W}}.
\end{align*}

To prove bordism invariance of $\Phi_\mathcal{A}$, assume that 
$$w=(W,(\mathcal{E}_{\mathcal{A}(\Gamma)},\mathcal{E}^{\prime}_{\mathcal{A}(\Gamma)} , E_{\C}, E_{\C}^{\prime} ,\alpha ),f)\sim_{bor} 0$$ 
in $\mathcal{S}^{geo}_0(\Gamma,\mathcal{A})$. By the definition of bordism, there is a cycle 
$$x=(\tilde{W},(\tilde{\mathcal{E}}_{\mathcal{A}(\Gamma)},\tilde{\mathcal{E}}^{\prime}_{\mathcal{A}(\Gamma)} , \tilde{E}_{\C}, \tilde{E}_{\C}^{\prime} ,\tilde{\alpha} ),\tilde{f})$$ 
(as in Lemma \ref{trivialdata}) such that $\partial W=\partial \tilde{W}$, $\tilde{f}=f$ and 
$$(\mathcal{E}_{\mathcal{A}(\Gamma)},\mathcal{E}^{\prime}_{\mathcal{A}(\Gamma)} , E_{\C}, E_{\C}^{\prime} ,\alpha )|_{\partial W}=(\tilde{\mathcal{E}}_{\mathcal{A}(\Gamma)},\tilde{\mathcal{E}}^{\prime}_{\mathcal{A}(\Gamma)} , \tilde{E}_{\C}, \tilde{E}_{\C}^{\prime} ,\tilde{\alpha} )|_{\partial\tilde{W}}.$$
By Lemma  \ref{trivialdata}, $\Phi_\mathcal{A}(x)=0$. Hence, Lemma \ref{gluingcyclesalongboundary} and  bordism invariance of the index on closed manifolds imply that 
$$\Phi_\mathcal{A}(w)=\Phi_\mathcal{A}(w\dot{\cup}x)=\Phi_\mathcal{A}(w\cup_{\partial W}x)=0,$$

To prove that the diagram \eqref{commutingsigmaoneandtwo} commutes, if $(M,\mathcal{E})$ is a cycle for $K_0^{geo}(pt;\mathcal{A}(\Gamma))$, then in $S^{h}_1(\sigma_1,\sigma_2)$
\[(\sigma_1-\sigma_2)_*(M,\mathcal{E})=(\emptyset,\emptyset,\emptyset,\emptyset,n),\]
where $n=\ind(M,\mathcal{E}\otimes_{\sigma_1} \C^k)-\ind(M,\mathcal{E}\otimes_{\sigma_2} \C^k)$. On the other hand, since $M$ is closed,
\[\Phi_\mathcal{A}\circ r(M,\mathcal{E})=\Phi(M,(\mathcal{E},M\times 0,\emptyset,\emptyset,\emptyset),\emptyset)=(\emptyset,\emptyset,\emptyset,\emptyset,n).\]

It is clear that $\delta=\delta_{\sigma_1,\sigma_2}\circ \Phi_\mathcal{A}$ since 
$$\delta_{\sigma_1,\sigma_2}\circ \Phi_\mathcal{A}(W,(\mathcal{E}_{\mathcal{A}(\Gamma)},\mathcal{E}^{\prime}_{\mathcal{A}(\Gamma)} , E_{\C}, E_\C^{\prime} ,\alpha ),f)=(\partial W,S_{\partial W}\otimes E_{\C},f)\dot{\cup}-(\partial W,S_{\partial W}\otimes E_{\C}',f).$$
\end{proof}

\subsection{The map $\Phi_\mathcal{A}$ for $C^*$-algebras}
\label{phiaforcstar}

Later, we will make use of a modification of the map $\Phi_{\mathcal{A}}$ that fits well with $C^*$-algebra completions. Informally, it maps to cycles with more ``analytic data"; somewhat more precisely, it maps into a well behaved class of cycles for $\tilde{\mathcal{S}}_1^h(\sigma_1,\sigma_2)$. Let $W$ be an even-dimensional compact spin$^c$ manifold with boundary, $f:\partial W\to B\Gamma$ continuous and $\xi$ a $K$-theory cocycle relative to assembly for free $\Gamma$-actions (see the Introduction or in more detail \cite[Definition 1.8]{paperI}). 

\begin{define}
\label{decorations}
Assume that $\Xi_0=(D_{\mathcal{E}},D_{\mathcal{E}'},D_{E},D_{E'})$ is a choice of Dirac operators (see Definition \ref{connectionsonrelcycles}) for $\xi=(\mathcal{E}_{C^*_{{\bf full}}(\Gamma)},\mathcal{E}^{\prime}_{C^*_{{\bf full}}(\Gamma)} , E_{\field{C}}, E_{\field{C}}^{\prime} ,\alpha )$. Let $A=(A^\mathcal{E},A^{\mathcal{E}'},A)$ be a triple of smoothing operators in the Mischenko-Fomenko calculus:
\begin{align*}
A^{\mathcal{E}}\in \Psi^{-\infty}_{C^*_{{\bf full}}(\Gamma)}&(\partial W,\mathcal{E}\otimes S_{\partial W}),\;A^{\mathcal{E}'}\in \Psi^{-\infty}_{C^*_{{\bf full}}(\Gamma)}(\partial W,\mathcal{E}'\otimes S_{\partial W})\\
\mbox{and}\quad&A\in \Psi^{-\infty}_{C^*_{{\bf full}}(\Gamma)}(\partial W\dot{\cup}-\partial W,(E'\otimes S_{\partial W}\otimes f^*\mathcal{L})\,\dot{\cup}-(E\otimes S_{\partial W}\otimes f^*\mathcal{L}))
\end{align*}
such that the boundary operators
\begin{align*}
D^{\partial W}_{\mathcal{E}}+A^\mathcal{E},\;D^{\partial W}_{\mathcal{E}'}+A^{\mathcal{E}'},\;\left(D^{\partial W}_{E'\otimes f^*\mathcal{L}}\dot{\cup}-D^{\partial W}_{E\otimes f^*\mathcal{L}}\right)+A
\end{align*}
are all invertible. We say that the data $\Xi=(\xi,\Xi_0,A)$ is a decoration of $\xi$. A decorated $K$-theory cocycle relative to assembly is a decoration of a $K$-theory cocycle relative to assembly.

Whenever $(W,\xi,f)$ is a cycle for $\mathcal{S}_0^{geo}(\Gamma,C^*_{{\bf full}})$ and $\Xi$ is a decoration of $\xi$, we say that $(W,\Xi,f)$ is a \emph{decorated} cycle for $\mathcal{S}_0^{geo}(\Gamma,C^*_{{\bf full}})$.
\end{define}

\begin{lemma}
Any $K$-theory cocycle relative to assembly on $(W,\partial W,f)$ admits a decoration.
\end{lemma}

\begin{proof}
Let $(W,(\mathcal{E}_{C^*_{\bf full}(\Gamma)},\mathcal{E}_{C^*_{\bf full}(\Gamma)}',E_\C,E_\C',\alpha),f)$ be a cycle for $\mathcal{S}_*^{geo}(\Gamma,C^*_{{\bf full}})$. As such, each of the cycles $(\partial W,\mathcal{E}_{C^*_{\bf full}(\Gamma)})$, $(\partial W,\mathcal{E}_{C^*_{\bf full}(\Gamma)}')$ and $(\partial W,E_\C\otimes f^*\mathcal{L})\dot{\cup}-(\partial W,E_\C'\otimes f^*\mathcal{L})$ for $K_*^{geo}(pt,C^*_{{\bf full}}(\Gamma))$ are nullbordant; in particular, their higher index vanish. Given a choice of Dirac operators for $\xi:=(\mathcal{E}_{C^*_{\bf full}(\Gamma)},\mathcal{E}_{C^*_{\bf full}(\Gamma)}',E_\C,E_\C',\alpha)$, Theorem \ref{leichpiazzass} implies the existence of smoothing operators. 
\end{proof}

\begin{define}[The higher Atiyah-Patodi-Singer index of a decorated cycle]
Let $(W,\Xi,f)$ be a decorated cycle for $\mathcal{S}_0^{geo}(\Gamma,C^*_{{\bf full}})$. Consider the Dirac operator $\bar{D}$ on the $C^*_{{\bf full}}$-Clifford bundle 
$$\mathfrak{S}:=\left(\pi^*\left(\mathcal{E}_{C^*_{\bf full}(\Gamma)}|_{\partial W}\right)\oplus \pi^*E_\C'\otimes f^*\mathcal{L}\right)\otimes S_{\partial W\times [0,1]}\to \partial W\otimes [0,1]$$ 
defined as in \eqref{barofnabla} (see page \pageref{barofnabla}) from $\Xi_0$. The $C^*_{\bf full}(\Gamma)$-Clifford bundle $\mathfrak{S}$ is graded since $W$ is even-dimensional and using the identification 
$$\mathfrak{S}^+|_{\partial W\times \{i\}}=(-1)^i \left(\mathcal{E}_{C^*_{\bf full}(\Gamma)}|_{\partial W}\oplus E_\C'\otimes f^*\mathcal{L}\right)\otimes S_{\partial W}\quad\mbox{for}\quad i=0,1,$$ 
we define the smoothing operator 
\begin{align*}
\bar{A}:=(\id\dot{\cup}\alpha^*)^{-1}&\left(\left(A^\mathcal{E}\dot{\cup} -A^{\mathcal{E}'}\right)\oplus A\right)(\id\dot{\cup}\alpha^*)\\
&\in\Psi^{-\infty}_{C^*_{{\bf full}}(\Gamma)}(\partial W\dot{\cup}-\partial W,\mathfrak{S}^+|_{\partial W}).
\end{align*}
The higher Atiyah-Patodi-Singer index of $(W,\Xi,f)$ is defined as
\begin{align*}
\ind_{APS}(W,\Xi,f):=\ind_{APS}(D^W_{\mathcal{E}},A^\mathcal{E})&-\ind_{APS}(D^W_{\mathcal{E}'},A^{\mathcal{E}'})\\
&+\ind_{APS}(\bar{D}^{\partial W\times [0,1]},\bar{A})\in K_0(C^*_{{\bf full}}(\Gamma)).
\end{align*}
\end{define}

\begin{remark}
The higher Atiyah-Patodi-Singer index of $(W,\Xi,f)$ depends (to a very large extent) on the choice of decoration $\Xi$; for the trivial group, this fact can be seen from Equation \eqref{apsdifferentp} (found on page \pageref{apsdifferentp}).
\end{remark}

\begin{prop}
Following the notation of Definition \ref{decorations}, the association 
\begin{align}
&\tilde{\Phi}_{C^*_{\bf full}}^0(W,\Xi,f):=
\label{tildephide}\\
\nonumber
 (\partial W\dot{\cup}-\partial W ,E_\C\otimes S_{\partial W}\dot{\cup}-&E_\C'\otimes S_{\partial W},f\dot{\cup}f,D_{E}\dot{\cup}-D_{E'},A,(\sigma_1-\sigma_2)_*\ind_{APS}(W,\Xi,f)),
\end{align}
induces a well defined map $\tilde{\Phi}_{C^*_{{\bf full}}}:\mathcal{S}_0^h(\Gamma;C^*_{{\bf full}})\to \tilde{S}_1^h(\sigma_1,\sigma_2)$ fitting into a commutative diagram:
\begin{center}
$$\xymatrix{
 &\mathcal{S}_0^{geo}(\Gamma,C^*_{{\bf full}})\ar[ddr]^{\tilde{\Phi}_{C^*_{{\bf full}}}}\ar[ddl]_{\Phi_{C^*_{{\bf full}}}} &  
 \\ \\
\mathcal{S}_1^h(\sigma_1,\sigma_2)\ar[rr]^{\huge \sim}& &\tilde{\mathcal{S}}_1^h(\sigma_1,\sigma_2)
}, $$
\end{center}
where the bottom row is the isomorphism of Lemma \ref{relwwotilde}.
\end{prop}

\begin{proof}
The result follows from Theorem \ref{phiacommutingdiagram} once we prove that for any decorated cycle $(W,\Xi,f)$ there is, modulo the disjoint union/direct sum relation, a decorated bordism $\tilde{\Phi}^0_{C^*_{{\bf full}}}(W,\Xi,f)\sim_{bor} \Phi_{C^*_{{\bf full}}}(W,\xi,f)$ whenever the cycle $\Phi_{C^*_{{\bf full}}}(W,\xi,f)$ is defined using the connection data in the decoration $\Xi$ of $\xi$ (see Definition \ref{defOfPhi}). We note that by functoriality of Atiyah-Patodi-Singer indices,
\begin{align*}
(\sigma_1-\sigma_2)_*&\ind_{APS}(W,\Xi,f)=\\
&\ind_{APS}(D^W_{G_1},A^\mathcal{E}_1)-\ind_{APS}(D^W_{G_2},A^\mathcal{E}_2)\\
-&\ind_{APS}(D^W_{G_1'},A^{\mathcal{E}'}_1)+\ind_{APS}(D^W_{G_2'},A^{\mathcal{E}'}_2)\\
+&\ind_{APS}(\bar{D}^{\partial W\times [0,1]}_{1},\bar{A}_1)-\ind_{APS}(\bar{D}^{\partial W\times [0,1]}_{2},\bar{A}_2)=
n-m,
\end{align*}
where the Clifford connections $\nabla_{G_1}$ et cetera and $n$ are as in Definition \ref{defOfPhi} while $m$ is as in Proposition \ref{bordingoutsmoothing}. Proposition \ref{bordingoutsmoothing} implies that 
\begin{align*}
\tilde{\Phi}_{C^*_{{\bf full}}}^0&(W,\Xi,f)=\\
&(\partial W\dot{\cup}-\partial W ,E_\C\otimes S_{\partial W}\dot{\cup}-E_\C'\otimes S_{\partial W},f\dot{\cup}f,D_{E}\dot{\cup}-D_{E'},A,n-m)\sim_{bor}\\
&(\partial W ,E_\C\otimes S_{\partial W},f,D_{E},0,n)\dot{\cup}-(\partial W,E_\C'\otimes S_{\partial W},f,D_{E'},0,0)=\Phi_{C^*_{{\bf full}}}(W,\xi,f).
\end{align*}
\end{proof}

\begin{remark}
\label{sigmaonetworem}
If $x\in K_*(C^*_{{\bf full}}(\Gamma))$ is the image of the cycle $(M,E,f)$ for $K_*^{geo}(B\Gamma)$ under assembly, Atiyah's $L^2$-index theorem for coverings and the Atiyah-Singer index theorem implies that $(\sigma_1-\sigma_2)_*(x)=0$. Hence, if the assembly map $\mu:K_*^{geo}(B\Gamma)\to K_*(C^*_{{\bf full}}(\Gamma))$ is surjective then $(\sigma_1-\sigma_2)_*=0$ as a map on $K_*(C^*_{{\bf full}}(\Gamma))$.  
\end{remark}

\section{Relation between $\rho\circ \Phi_\mathcal{A}$ and the geometric $\rho$-invariant of \cite{paperI}}
\label{relatingtoindaleph}

 In \cite[Section $5$]{paperI}, a map, $\mathcal{S}_0^{geo}(\Gamma,C^*_{{\bf full}})\to \field{R}$, was constructed on the level of geometric cycles. In this section, we compare this map with the map $\rho_{\sigma_1,\sigma_2}\circ \Phi_{C^*_{{\bf full}}}:\mathcal{S}_0^{geo}(\Gamma,C^*_{{\bf full}})\to \field{R}$. We begin by recalling the construction from \cite[Section $5$]{paperI}.

Let $N$ denote a II$_1$-factor (i.e. a finite von Neumann algebra factor with $K_0(N)\cong\field{R}$); we identify $K_0(N)$ with $\R$ via the unique normal faithful tracial state of $N$. Following the notation of \cite[Definition $5.1$]{paperI} and \cite[Example $5.2$]{paperI}, we construct the following data from the unitary rank $k$ representations $\sigma_1$ and $\sigma_2$ of $\Gamma$. As above, we let $E_i:=E\Gamma\times_{\sigma_i}\C^k\to B\Gamma$. We \emph{choose} an isomorphism $\phi:E_1\otimes N\to E_2\otimes N$ of $N$-bundles on $B\Gamma$. Such an isomorphism is constructed along the lines of \cite[Proposition $5.2$]{antazzska}. The existence of such an isomorphism follows from the fact that $E_1$ and $E_2$ are flat. 

We let $\aleph_0:=(E_1,E_2,\phi)$ denote the associated cocycle for $K^1(B\Gamma; \rz)$, for more on $K$-theory with coefficients in $\rz$ see \cite{antazzska, Bas, DeeRZ, Kar}. The pairing between $K$-homology and $K$-theory with coefficients in $\rz$ is discussed in the context of geometric cycles in \cite[Section 6.2]{DeeRZ}. The data $\aleph_0$ was required in \cite[Definition $5.1$]{paperI} to be refined by further data, which in this case is canonically associated with the data $(\aleph_0,\sigma_1,\sigma_2)$ and as such we denote this data by $\aleph$. We remark that in the general setup considered in \cite{paperI} it was the $M_k(\C)$-bundles $\mathcal{L}\otimes_{\sigma_i}M_k(\C)$ and not the vector bundles $E_i=\mathcal{L}\otimes_{\sigma_i}\C^k$ that was used in $\aleph$. This difference does not affect the comparison carried out in this section.

The map $\ind^\R_\aleph:\mathcal{S}_*^{geo}(\Gamma,C^*_{{\bf full}})\to K_*(pt;N)=\R$ was defined in \cite[Proposition $5.5$]{paperI} on the level of cycles by
\[(W,\xi,f)\mapsto (Z,\alpha_{W,f}^\aleph(\xi)),\]
where $Z:=W\cup_{\partial W}\partial W\times [0,1]\cup_{\partial W} -W$ and $\alpha_{W,f}^\aleph:K^0(W,\partial W;\mu_\mathcal{L})\to K^0(Z;N)$ is a map canonically constructed from $W$, $f$ and $\aleph$. The reader is referred to \cite[Section $5$]{paperI} for the precise construction. 

We will need to understand $\alpha_{W,f}^\aleph$ on an explicit level modulo terms producing null-bordant cycles. Assume that $\xi=(\mathcal{E}_{C^*_{\bf full}(\Gamma)},\mathcal{E}_{C^*_{\bf full}(\Gamma)}',E_\C,E_\C',\alpha)$ is a cycle for $K^0(W,\partial W;\mu_\mathcal{L})$. For notational simplicity, let
\[\tilde{\mathcal{E}}_{C^*_{\bf full}(\Gamma)}:=\pi^*(E_\C\otimes f^*\mathcal{L})\to \partial W\times[0,1]\quad\mbox{and}\quad \tilde{\mathcal{E}}'_{C^*_{\bf full}(\Gamma)}:=\pi^*(E'_\C\otimes f^*\mathcal{L})\to \partial W\times[0,1]\]
where $\pi$ denotes the projection map, $\partial W \times [0,1] \rightarrow \partial W$. 

Throughout this subsection, we abuse notation by letting $C^*_{\bf full}(\Gamma)^k$ denote the bundle corresponding to $k$ copies of the unit class in $C(M,C^*_{{\bf full}}(\Gamma))$ for a manifold $M$ that will be clear from its context. Similarly, we let $1^k$ denote the bundle corresponding to $k$ copies of the unit class in $C(M)$. We choose a complementary $C^*_{{\bf full}}(\Gamma)$-bundle  $\mathcal{E}^\perp_{C^*_{\bf full}(\Gamma)}\to W$ for $\mathcal{E}_{C^*_{\bf full}(\Gamma)}$; that is, there is an isomorphism 
$$\mathcal{E}_{C^*_{\bf full}(\Gamma)}\oplus \mathcal{E}^\perp_{C^*_{\bf full}(\Gamma)}\cong C^*_{\bf full}(\Gamma)^n$$ 
for some $n$. We also choose a complementary $C^*_{{\bf full}}(\Gamma)$-bundle $\tilde{\mathcal{E}}^\perp_{C^*_{\bf full}(\Gamma)}\to \partial W\times [0,1]$ for $\tilde{\mathcal{E}}_{C^*_{\bf full}(\Gamma)}$; that is, there is an isomorphism 
$$\tilde{\mathcal{E}}_{C^*_{\bf full}(\Gamma)}\oplus\tilde{\mathcal{E}}^\perp_{C^*_{\bf full}(\Gamma)}\cong  C^*_{\bf full}(\Gamma)^{\tilde{n}}$$ 
for some $\tilde{n}$. With the isomorphism $\alpha$ and the choice of complementary bundles in hand, we obtain an isomorphism of $C^*_{{\bf full}}(\Gamma)$-bundles on $\partial W\times[0,1]$:
\[\hat{\alpha}:\tilde{\mathcal{E}}^\perp_{C^*_{\bf full}(\Gamma)}\oplus \tilde{\mathcal{E}}_{C^*_{\bf full}(\Gamma)}'\oplus C^*_{\bf full}(\Gamma)^n\xrightarrow{\sim} \mathcal{E}'_{C^*_{\bf full}(\Gamma)}\oplus \mathcal{E}^\perp_{C^*_{\bf full}(\Gamma)}\oplus C^*_{\bf full}(\Gamma)^{\tilde{n}}.\]
The isomorphism $\hat{\alpha}$ induces isomorphisms over $\partial W\times \{0\}$ for $i=1,2$:
\[\hat{\alpha}_i:=(\sigma_i)_*(\hat{\alpha}):(E_\C\otimes f^*E_i)^\perp\oplus E_\C'\otimes f^*E_i\oplus 1^{nk}\xrightarrow{\sim}G_i'|_{\partial W}\oplus G_i^\perp|_{\partial W}\oplus 1^{nk}.\]
We note that the choices $\mathcal{E}^\perp_{C^*_{\bf full}(\Gamma)}$ and $\tilde{\mathcal{E}}^\perp_{C^*_{\bf full}(\Gamma)}$ gives a canonical choice of complement for the bundles $E_\C\otimes f^*E_i$ and $G_i$. The isomorphism $\phi$ induces an isomorphism of vector bundles over $\partial W$:
\begin{align*}
\phi_*:&\left((E_\C\otimes f^*E_1)^\perp\oplus E_\C'\otimes f^*E_1\oplus 1^{nk}\right)\otimes N\xrightarrow{\sim} \\
&\qquad\qquad\left((E_\C\otimes f^*E_2)^\perp\oplus E_\C'\otimes f^*E_2\oplus 1^{nk}\right)\otimes N.
\end{align*}

Let $\Sigma_\xi\to Z=W\cup_{\partial W} \partial W\times [0,1]\cup_{\partial W} -W$ denote the $N$-bundle obtained by gluing together 
\begin{align*}
(G_1&'|_{\partial W}\oplus G_1^\perp|_{\partial W}\oplus 1^{nk})\otimes N\to W \quad\mbox{along}\quad \hat{\alpha}_1\otimes \id_N\quad\mbox{with}\\
&\pi^*\left((E_\C\otimes f^*E_1)^\perp\oplus E_\C'\otimes f^*E_1\oplus 1^{nk}\right)\otimes N\to \partial W\times [0,1]\\
&\quad\mbox{which is glued along}\quad \hat{\alpha}_2\circ \phi_*\quad\mbox{with}\quad \left[G_2'|_{\partial W}\oplus G_2^\perp|_{\partial W}\oplus 1^{nk}\right]\otimes N\to -W.
\end{align*}
The following Proposition follows from the constructions in \cite[Section $5$]{paperI}.

\begin{prop}
In the notation above, $\alpha_{W,f}^\aleph(\xi)-[\Sigma_\xi]\in\R[1]\subseteq K^0(Z;N)$.
\end{prop} 

Since the manifold $Z$ is nullbordant, $(Z,\alpha_{W,f}^\aleph(\xi))\sim_{bor} (Z,\Sigma_\xi)$ in $K_0^{geo}(pt;N)$. We now turn to the comparison of $\ind_\aleph^\R$ with $\rho_{\sigma_1,\sigma_2}\circ \Phi_{C^*_{{\bf full}}}$. First we recall that a II$_1$-factor $N$ is equipped with a unique faithful normal tracial state $\tau_N$ (see \cite[Corollary III$.2.5.8$]{blackadarbook}). Whenever $\mathcal{F}_N\to W$ is a smooth $N$-bundle equipped with a connection $\nabla_\mathcal{F}$, we use the notation $\ch_\tau(\nabla_\mathcal{F})$ for the associated Chern character, see more in \cite{Sch}. Whenever $Z$ is closed, the cohomology class of  $\ch_\tau(\nabla_\mathcal{F})$ only depend on $\mathcal{F}$ and we write $\ch_\tau(\mathcal{F})$ for the associated de Rham cohomology class. If $W=M\times [0,1]$, $\mathcal{F}=\mathcal{F}_0\times [0,1]\to W$ for a smooth $N$-bundle $\mathcal{F}_0\to M$ equipped with two connections $\nabla_{\mathcal{F}_0}$ and $\nabla_{\mathcal{F}_0}'$, we can equipp $\mathcal{F}$ with the connection $\nabla_\mathcal{F}:=t\nabla_{\mathcal{F}_0}+(1-t)\nabla_{\mathcal{F}_0}'$. We set
$$\cs_\tau(\nabla_{\mathcal{F}_0},\nabla_{\mathcal{F}_0}'):=\int_0^1 \ch_\tau(\nabla_\mathcal{F}).$$

\begin{define}
\label{csdef}
Let $(W,\xi,f)$ be a cycle for $\mathcal{S}_0^{geo}(\Gamma,\mathcal{A})$ equipped with Dirac operators $\Xi_0$ constructed from a connection $(\nabla_W,\nabla_{\mathcal{E}},\nabla_{\mathcal{E}'},\nabla_{E},\nabla_{E'})$ (see Definition \ref{connectionsonrelcycles}). We define the Clifford connections
\[\tilde{\nabla}_1:=\nabla_{(E\otimes E_1)^\perp\otimes N}\oplus \nabla_{E'\otimes E_1\otimes N}\oplus\nabla_{1^{nk}\otimes N},\]
\[\tilde{\nabla}_2:=\phi_*^{-1}\left(\nabla_{(E\otimes E_2)^\perp\otimes N}\oplus \nabla_{E'\otimes E_2\otimes N}\oplus\nabla_{1^{nk}\otimes N}\right),\]
where $\nabla_{1^{nk}}$ denotes the trivial connection on $1^{nk}$. The $\aleph$-Chern-Simons invariant of the cycle with connection $(W,\xi,\Xi_0,f)$ is given by
\[\cs_{\aleph}(W,\xi,\Xi_0,f):=\int_{\partial W} \cs_\tau\left(\tilde{\nabla}_1,\tilde{\nabla}_2\right) \wedge Td(\nabla_{\partial W}).\]
\end{define}

\begin{remark}
Suppose that $\xi$ is an easy cycle; that is, $\mathcal{E}'=E'=0$ and $\alpha:\mathcal{E}|_{\partial W}\xrightarrow{\sim} E \otimes f^*\mathcal{L}$. In this case, we can take the connection $\nabla_{\mathcal{E}}$ such that it coincides with the connection induced from $\nabla_E$ via $\alpha$ near $\partial W$. In this case, we have that
\[\cs_{\aleph}(W,\xi,\Xi_0,f)=-\int_{\partial W} \cs_\tau\left(\nabla_{E\otimes E_1},\phi_*^{-1}(\nabla_{E\otimes E_2})\right) \wedge Td(\nabla_{\partial W}).\]
\end{remark}

\begin{remark}
\label{asforvnbun}
Whenever $Z$ is an even-dimensional closed spin$^c$-manifold and $\mathcal{F}_N\to Z$ is an $N$-bundle, i.e. $(Z,\mathcal{F}_N)$ is a cycle for $K_0^{geo}(pt;N)$, its class in $K_0^{geo}(pt;N)$ is determined by the index $\ind_N(Z,\mathcal{F}_N)\in K_0(N)$. The trace, $\tau_N$, induces an isomorphism $(\tau_N)_*:K_0(N)\xrightarrow{\sim} \R$. Moreover, standard results from Mishchenko-Fomenko index theory (see for example \cite[Theorem 6.9]{Sch}) imply that
$$(\tau_N)_*\ind_N(Z,\mathcal{F}_N)=\int_Z\ch_\tau(\mathcal{F}_N)\wedge Td(Z).$$
\end{remark}

\begin{theorem}
\label{indandrho}
Assume that $\sigma_1,\sigma_2:\Gamma\to U(k)$ are representations and let $\aleph$ and  $\aleph_0$ be the data chosen in the introduction of this section. The $\aleph$-Chern-Simons invariant of cycles with connection is well defined and induces a map $\cs_{\aleph}:\mathcal{S}^{geo}_0(\Gamma,C^*_{{\bf full}})\to \R$ fitting into the commutative diagram:
\begin{center}
$$\xymatrix{
\mathcal{S}_0^{geo}(\Gamma,C^*_{{\bf full}})\ar[rr]^{\Phi_{C^*_{{\bf full}}}}\ar[ddr]_{(\tau_N)_*\circ \ind_\aleph^\R-\cs_{\aleph}}& &\mathcal{S}_1^h(\sigma_1,\sigma_2)\ar[ddl]^{\rho_{\sigma_1,\sigma_2}}
\\ \\ 
 &\R&  
}.$$
\end{center}
where
\begin{enumerate}
\item $\Phi_{C^*_{\bf full}}$ is the map in Definition \ref{defOfPhi};
\item $\rho_{\sigma_1, \sigma_2}$ is the map in Definition \ref{definitionofrho};
\item $\tau_N$, ${\rm ind}^{\field{R}}_{\aleph}$, and $cs_{\aleph}$ are the maps discussed in the definitions and remarks just before the statement of the Theorem.
\end{enumerate}
\end{theorem}
\begin{remark}
The reader should note that this theorem appeared in the Introduction. As we mentioned there, the map ${\rm ind}^{\field{R}}_{\aleph}$ should be viewed in analogy with the topological index, while the map $\rho_{\sigma_1, \sigma_2}$ should be viewed in analogy with the analytic index.
\end{remark}
\begin{proof}
We equip $Z$ with the Clifford connection constructed from $\nabla_W$ using the fact that $\nabla_W$ is of product type near $\partial W$. Remarks \ref{ncomputation} and \ref{asforvnbun} imply that the result follows upon showing that 
\small
\begin{align}
\nonumber
&\int_Z\ch_\tau[\Sigma_\xi]\wedge Td(Z)-\cs_{\aleph}(W,\xi,\Xi_0,f)=\int_W \left(\ch(\nabla_{G_1})-\ch(\nabla_{G_2})\right)\wedge Td(\nabla_W)-\\
\label{etavsindaleph}
&\qquad\int_W \left(\ch(\nabla_{G'_1})-\ch(\nabla_{G'_2})\right)\wedge Td(\nabla_W)+\int_{\partial W\times [0,1]} \left(\ch(\bar{\nabla}_1)-\ch(\bar{\nabla}_2)\right)\wedge Td(\nabla_{\partial W}).
\end{align}
\normalsize
To prove this fact, we construct a specific connection $\nabla^\Sigma$ on $\Sigma_\xi$. Let $[0,1]\times \partial W\cong U\subseteq W$ be a small cylindrical neighborhood of $\partial W$ and $[0,1/2]\times \partial W\cong U_0\subseteq U$ another even smaller cylindrical neighborhood. We can take 
\begin{align*}
\nabla^\Sigma|_{W\setminus \overline{U}}=\nabla_{G_1^\perp\otimes N}\oplus& \nabla_{G_1'\otimes N}\oplus \nabla_{1^{\tilde{n}}\otimes N}\quad\mbox{and}\\
& \nabla^\Sigma|_{U_0}=\hat{\alpha}_1^{-1}\left(\nabla_{(E\otimes E_1)^\perp\otimes N}\oplus \nabla_{E'\otimes E_1\otimes N}\oplus\nabla_{1^n\otimes N}\right).
\end{align*} 
On the other part (i.e., $-W\subseteq Z$) we require that 
\begin{align*}
\nabla^\Sigma|_{-W\setminus -\overline{U}}=\nabla_{G_2^\perp\otimes N}\oplus& \nabla_{G_2'\otimes N}\oplus \nabla_{1^{\tilde{n}}\otimes N}\quad\mbox{and}\\
& \nabla^\Sigma|_{-U_0}=\hat{\alpha}_2^{-1}\left(\nabla_{(E\otimes E_2)^\perp\otimes N}\oplus \nabla_{E'\otimes E_2\otimes N}\oplus\nabla_{1^n\otimes N}\right).
\end{align*} 
On $\partial W\times [0,1]$, we set the conditions that 
\begin{align*}
\nabla^\Sigma&=\hat{\alpha}_1^{-1}\tilde{\nabla}_1\quad \mbox{near}\quad \partial W\times \{0\}\quad\mbox{and}\quad  \nabla^\Sigma =\hat{\alpha}_1^{-1}\tilde{\nabla}_2\quad \mbox{near}\quad \partial W\times \{1\}.
\end{align*} 
For this choice of connections,
\begin{align*}
&\int_{W\dot{\cup}-W}\ch_\tau[\Sigma_\xi]\wedge Td(Z)=\int_W \left(\ch(\nabla_{G_1})-\ch(\nabla_{G_2})\right)\wedge Td(\nabla_W)-\\
&\qquad\int_W \left(\ch(\nabla_{G'_1})-\ch(\nabla_{G'_2})\right)\wedge Td(\nabla_W)+\int_{\partial W\times [0,1]} \left(\ch(\bar{\nabla}_1)-\ch(\bar{\nabla}_2)\right)Td(\nabla_{\partial W}),
\end{align*}
and
\begin{align*}
\int_{\partial W\times [0,1]}\ch_\tau[\Sigma_\xi]\wedge Td(Z)&=\int_{\partial W} \cs_\tau\left( \hat{\alpha}_1^{-1}\tilde{\nabla}_1, \hat{\alpha}_1^{-1}\tilde{\nabla}_2\right) \wedge Td(\nabla_{\partial W})=\\
&=\int_{\partial W} \cs_\tau\left(\tilde{\nabla}_1,\tilde{\nabla}_2\right) \wedge Td(\nabla_{\partial W})=\cs_{\aleph}(W,\xi,\Xi_0,f).
\end{align*}
The identity \eqref{etavsindaleph} (and thus also the theorem) follows.
\end{proof}

\begin{remark}
The reader should compare Theorem \ref{indandrho} to \cite[Theorem $5.4$]{antazzska}.
\end{remark}

Suppose that $\aleph$ and $\aleph'$ are two different choices of data above, i.e. we choose two different isomorphisms $\phi:E_1\otimes N\to E_2\otimes N$ and $\phi':E_1\otimes N' \to E_2\otimes N'$ for two II$_1$-factors $N$ and $N'$. We define the $N\bar{\otimes} N'$-bundle $\ell_{\aleph,\aleph'}\to B\Gamma\times S^1$ by gluing together $E_1\otimes N\bar{\otimes} N'\times [0,1/2]\to B\Gamma\times [0,1/2]$ with $E_2\otimes N\bar{\otimes} N'\times [1/2,1]\to B\Gamma\times [1/2,1]$ in $1/2$ along $\phi\otimes \id_{N'}$ and  the fiber over $0$ with that over $1$ along $\id_N\otimes \phi'$. We conclude the following Proposition comparing the Chern-Simons term of two different choices $\aleph$ and $\aleph'$ as above.

\begin{prop}
We use the notation introduction in this section (in particular, in the previous paragraph and Remark \ref{asforvnbun}). Then, one can factor the difference 
\[(\tau_N)_*\circ \ind_\aleph^\R-(\tau_{N'})_*\circ\ind_{\aleph'}^\R=(\tau_N\bar{\otimes}\tau_{N'})_*\circ \gamma_{\aleph,\aleph'}\circ \delta:\mathcal{S}_0^{geo}(\Gamma,C^*_{\bf full})\to \R,\]
where $\gamma_{\aleph,\aleph'}:K_1^{geo}(B\Gamma)\to K_0^{geo}(pt;N\bar{\otimes}N')$ is given on the level of cycles by
\[\gamma_{\aleph,\aleph'}(M,E,f):=(M\times S^1,E\otimes f^*\ell_{\aleph,\aleph'}).\]
\end{prop}

We note that if $[\aleph_0]=[\aleph_0']\in K^1(B\Gamma;\rz)$, \cite[Theorem $5.6$]{paperI} guarantees that $(\tau_N\bar{\otimes}\tau_{N'})_*\circ \gamma_{\aleph,\aleph'}\circ \delta:\mathcal{S}_0^{geo}(\Gamma,C^*_{\bf full})\to \Z$.

\section{Vanishing results for relative $\eta$-invariants}
\label{sectionvanishing}

The main application of this paper is the connection between the cycles relative to the assembly map and relative $\eta$-invariants (in particular, stable relative $\eta$-invariants). Such connections have previously been used to establish vanishing of stable relative $\eta$-invariants for cycles vanishing under the full assembly map assuming that $\Gamma$ is a torsion-free group with the full Baum-Connes property (see for instance \cite{HReta, Kescont, PS}). 

The motivation for this sphere of problems is the following: assume that $(M,S_{\C\ell},f)$ is a cycle for $K_1^h(B\Gamma)$ that vanishes under assembly. This situation occurs for instance when $M$ is spin and has positive scalar curvature, with $S_{\C\ell}$ being the spinor bundle, or if there in a suitable sense exists a homotopy equivalence between two cycles $(M',S'_{\C\ell},f')$ and $(M'',S''_{\C\ell},f'')$ in which case the cycle $(M,S_{\C\ell},f)=(M',S'_{\C\ell},f')\dot{\cup}(-M'',-S''_{\C\ell},f'')$ vanishes under assembly. Following \cite[Theorem 3.8]{paperI}, the class associated to the cycle, $(M,S_{\C\ell},f)$, must be the boundary of a cycle $(W,\xi,f)$ for $\mathcal{S}_{0}^{geo}(\Gamma;C^*_{{\bf full}})$ under the isomorphism $K_1^{geo}(B\Gamma)\xrightarrow{\sim} K_1^h(B\Gamma)$. If $C^*_{{\bf full}}(\Gamma)$ has the Baum-Connes property, the map $\Phi_{C^*_{{\bf full}}}$ of Theorem \ref{phiacommutingdiagram} vanishes. After applying the relative $\eta$-invariant on $\mathcal{S}_1^{h}(\sigma_1,\sigma_2)$ we arrive at the identity 
$$\rho_{\sigma_1,\sigma_2}\circ\Phi_{C^*_{{\bf full}}}(W,\xi,f)=n+\rho_{\sigma_1,\sigma_2}(M,S_{\C\ell},f,D)=0.$$
However, there is a large freedom in choosing the cycle $(W,\xi,f)$ and vanishing results for relative $\eta$-invariants can only be obtained if it is possible to choose $(W,\xi,f)$ in such a way that $n=0$.

\subsection{The $(\sigma_1,\sigma_2)$-relative index problem}

Many results in this section are based on solving a certain type of index problem. The index problem can be formulated in two ways. One is suited for Banach algebras and the relative $\eta$-invariant and a second, which we treat in the next subsection, that (with currently available techniques) is only available for $C^*$-algebras; it is also better suited for treating stable relative $\eta$-invariants. 

\begin{define}[The $(\sigma_1,\sigma_2)$-relative index problem]
\label{sigmaproblem}
Assume that $\mathcal{A}(\Gamma)$ is a Banach algebra completion of $\C[\Gamma]$ such that $\sigma_1$ and $\sigma_2$ extend to continuous homomorphisms $\mathcal{A}(\Gamma)\to M_k(\C)$. The $(\sigma_1,\sigma_2)$-relative index problem for $(M,S_{\C\ell},g,D)$ (with respect to $\mathcal{A}(\Gamma)$), where $(M,S_{\C\ell},g)$ is a cycle for $K_1(B\Gamma)$ and $D$ is a choice of Dirac operator on $S_{\C\ell}$, is to find a cycle $(W,\xi,f)$ for $\mathcal{S}^{geo}_{0}(\Gamma,\mathcal{A})$ such that 
$$\Phi_\mathcal{A}(W,\xi,f)=[M,S_{\C\ell},g,D,0]\quad\mbox{in}\quad \mathcal{S}^{h}_1(\sigma_1,\sigma_2).$$
\end{define}

\begin{remark}
A solution to the $(\sigma_1,\sigma_2)$-relative index problem is never unique. A disjoint union by any cycle $(M,\mathcal{E})$ for $K_0^{geo}(pt;\mathcal{A}(\Gamma))$ such that $(\sigma_1-\sigma_2)_*(M,\mathcal{E})=0$ does not alter the boundary of a cycle for $\mathcal{S}^{geo}_{0}(\Gamma,\mathcal{A})$. 
\end{remark}

\begin{prop}
\label{indexsolutionrhovanishes}
Assume that $\Gamma$ is a torsion-free group and that $\mathcal{A}(\Gamma)$ (as in Definition \ref{sigmaproblem}) has the Baum-Connes property. If $(M,S_{\C\ell},g,D)$ admits a solution to the $(\sigma_1,\sigma_2)$-relative index problem, 
$$\rho_{\sigma_1,\sigma_2}(M,S_{\C\ell},g,D)=0.$$
\end{prop}

\begin{proof}
If $\mathcal{A}(\Gamma)$ has the Baum-Connes property, then $\mathcal{S}_0^{geo}(\Gamma,\mathcal{A})=0$; hence, 
\[\rho_{\sigma_1,\sigma_2}\circ \Phi_\mathcal{A}(W,\xi,f)=\rho_{\sigma_1,\sigma_2}(M,S_{\C\ell},g,D)=0\]
if $(W,\xi,f)$ solves the $(\sigma_1,\sigma_2)$-relative index problem for $(M,S_{\C\ell},g,D)$.
\end{proof}

\begin{remark}
\label{obstructingtheindexproblem}
The obvious topological obstruction to solving the $(\sigma_1,\sigma_2)$-relative index problem for $(M,S_{\C\ell},g,D)$ is that its image under assembly vanishes (i.e., $\mu_{\mathcal{A}}(M,S_{\C\ell},g)=0$ in $K_1(\mathcal{A}(\Gamma))$). Unfortunately, the $(\sigma_1,\sigma_2)$-relative index problem does in general not admit a solution even when the assembly of the cycle vanishes. This follows from Proposition \ref{indexsolutionrhovanishes} and the examples in \cite[Section $15$]{PS}.
\end{remark}

\subsection{The \emph{stable} $(\sigma_1,\sigma_2)$-relative index problem}
\label{stabelsubsection}

To remedy the problem discussed in Remark \ref{obstructingtheindexproblem}, we modify the $(\sigma_1,\sigma_2)$-relative index problem slightly. The drawback is that in the current state of higher Atiyah-Patodi-Singer theory, we can only consider $C^*$-algebra coefficients onto which $\sigma_1$ and $\sigma_2$ extend continuously.

\begin{define}[The stable $(\sigma_1,\sigma_2)$-relative index problem]
\label{stablesigmaproblem}
A decorated cycle $(W,\Xi,f)$ for $\mathcal{S}^{geo}_0(\Gamma,C^*_{{\bf full}})$ solves the stable $(\sigma_1,\sigma_2)$-relative index problem for the cycle $(M,S_{\C\ell},g,D)$ with Dirac operator if there is an even-dimensional oriented vector bundle $V\to M$ and a smoothing operator $A\in \Psi^{-\infty}(M^V,S_{\C\ell}^V\otimes (g^V)^*\mathcal{L})$ such that $D^V_{\mathcal{L}}+A$ is invertible and 
$$\tilde{\Phi}_{C^*_{\bf full}(\Gamma)}(W,\Xi,f)=[M^V,S_{\C\ell}^V,g^V,D^V,A,0]\quad\mbox{in}\quad \tilde{\mathcal{S}}^{h}_1(\sigma_1,\sigma_2)$$
\end{define}

\begin{prop}
\label{indexsolutionstablerhovanishes}
Suppose that the full $C^*$-completion of a torsion-free group $\Gamma$ has the Baum-Connes property. Then, whenever $(M,S_{\C\ell},g,D)$ admits a positive solution to the stable $(\sigma_1,\sigma_2)$-relative index problem, 
$$\rho^s_{\sigma_1,\sigma_2}(M,S_{\C\ell},g,D)=0.$$
\end{prop}

\begin{proof}
Using \cite[Proposition $6.6$]{HReta} we can assume that $V$ is the zero bundle. Note that the Baum-Connes property implies that $\mathcal{S}^{geo}_{0}(\Gamma,C^*_{{\bf full}})=0$ so $\rho_{\sigma_1,\sigma_2}\circ \tilde{\Phi}_{C^*_{{\bf full}}}=0$. Then, Proposition \ref{definitionofrho} implies that 
\[\rho_{\sigma_1,\sigma_2}\circ \tilde{\Phi}_{C^*_{{\bf full}}}(W,\Xi,f)=\rho_{\sigma_1,\sigma_2}(M,S_{\C\ell},g,D,A)=\rho^s_{\sigma_1,\sigma_2}(M,S_{\C\ell},g,D)=0\]
where $(W,\Xi,f)$ solves the stable $(\sigma_1,\sigma_2)$-relative index problem for $(M,S_{\C\ell},g,D)$; we note that the middle equality follows from Lemma \ref{stablerhoinvariantlemma}.
\end{proof}

\begin{theorem}
\label{solvingindexproblem}
Let $\Gamma$ be a discrete group and $(M,S_{\C\ell},f,D)$ be a $K_1^h(B\Gamma)$-cycle equipped with a Dirac operator. Then, this cycle admits a solution to the stable $(\sigma_1,\sigma_2)$-relative index problem if and only if $\mu_{C^*_{{\bf full}}}(M,S_{\C\ell},f)=0$.
\end{theorem}

Before proving this Theorem, we make a few preparatory remarks. 

\begin{remark}
The ``only if"-part of the proof of Theorem \ref{solvingindexproblem} is easy: if $(M,S_{\C\ell},f,D)$ admits a solution to the stable $(\sigma_1,\sigma_2)$-relative index problem, $(M,S_{\C\ell},f)$ is in the image of 
$$\mathcal{S}^{geo}_0(\Gamma,C^*_{{\bf full}})\xrightarrow{\delta} K_1^{geo}(B\Gamma)\xrightarrow{\sim} K_1^h(B\Gamma)$$ 
hence, by \cite[Theorem 3.8]{paperI}, $\mu_{C^*_{{\bf full}}}(M,S_{\C\ell},f)=0$. 

For the converse statement, it suffices to prove that whenever $(M,S_{\C\ell},f,D)$ is a $K_1^h(B\Gamma)$-cycle with a choice of Dirac operator such that $\mu(M,S_{\C\ell},f)=0$ in $K_1(C^*_{{\bf full}}(\Gamma))$, there exists:
\begin{enumerate}
\item an even-dimensional oriented vector bundle $V\to M$ such that $M^V$ can be equipped with a spin$^c$-structure;
\item a decorated cycle $(W,\Xi,f^V)$ for $\mathcal{S}_0^{geo}(\Gamma,C^*_{{\bf full}})$ with $\partial W=M^V$;
\item a null-bordant cycle $x$ for $\mathcal{S}_0^{geo}(\Gamma,C^*_{{\bf full}})$;
\item a cycle $y\in K_0^{geo}(pt;C^*_{{\bf full}}(\Gamma))$;
\end{enumerate}
such that
\[\tilde{\Phi}_{C^*_{{\bf full}}}\left((W,\Xi,f^V)\dot{\cup}r(y)\right)=(M^V,S^V_{\C\ell},f^V,D^V,A,0)\dot{\cup}\tilde{\Phi}_{C^*_{{\bf full}}}(x),\]
on the level of cycles modulo disjoint union/direct sum. We note here the subtlety that vector bundle modification on the right hand side is carried out in the oriented model (see Definition \ref{vectorbundlemodificationofrelativecycles}).
\end{remark}

\begin{remark}
In the proof of Theorem \ref{solvingindexproblem}, we carry out vector bundle modification in the geometric models $K_1^{geo}$ as well as $\mathcal{S}_0^{geo}$ using the \emph{Bott class} rather than the \emph{Bott bundle}, the latter approach lying closer to the vector bundle modification used in Subsection \ref{sectionorientedmodel}. The motivation for the use of the Bott class is that we need the notion of normal bordism; this notion requires the usage of cycles with $K$-theory data and modification using the Bott class. The reader can find the relevant details in \cite[Chapter $4.5$]{Rav} and \cite[Section $2$ and $3$]{paperI}.
\end{remark}

\begin{proof}[Proof of Theorem \ref{solvingindexproblem}]
For notational simplicity we assume $M$ to be connected. The proof for non-connected $M$ is carried out analogously (only the fiber dimensions of the vector bundles which one modifies by can vary on each connected component). The manifold data in the vector bundle modified cycle $(M,S_{\C\ell},f,D)^{T^*M}$ admits a canonical spin$^c$-structure, so we can assume that $M$ has a spin$^c$-structure $S_M\to M$ and that $S_{\C\ell}=S_M\otimes E_\C$ for some vector bundle $E_\C\to M$. 

Since $(M,E_\C\otimes f^*\mathcal{L})\sim 0$ in $K_1^{geo}(pt;C^*_{\bf full}(\Gamma))$ and the equivalence relation defining $K_1^{geo}(pt;C^*_{{\bf full}}(\Gamma))$ coincides with normal bordisms (see \cite[Chapter $4.5$]{Rav}), there is a $2k$-dimensional spin$^c$-normal bundle $V\to M$, a spin$^c$-manifold $W$ with boundary $\partial W=M^V$ and $C^*_{{\bf full}}(\Gamma)$-bundles $\mathcal{E}_{C^*_{{\bf full}}(\Gamma)},\mathcal{E}_{C^*_{{\bf full}}(\Gamma)}'\to W$ such that
\begin{align}
\label{partialeeee}
\partial &(W,[\mathcal{E}_{C^*_{{\bf full}}(\Gamma)}]-[\mathcal{E}_{C^*_{{\bf full}}(\Gamma)}'])\\
\nonumber
&=(M,[E_\C\otimes f^*\mathcal{L}])^V=(M^V,[E_\C^V\otimes (f^V)^*\mathcal{L}]-2^k[\pi^*_VE_\C\otimes (f^V)^*\mathcal{L}]),
\end{align}
on the level of isomorphism classes of cycles for $K_1^{geo}(pt;C^*_{{\bf full}}(\Gamma))$ with $K$-theory data. Here $E^V_\C$ is defined using the $2^k$-dimensional Bott bundle $Q^V\to M^V$ (associated with the spin$^c$-structure) by setting $E^V_\C:=\pi^*_VE\otimes Q_V$ (i.e., $E^V_\C$ is the vector-bundle modification of $E_\C$ in $K_1^{geo}$ using the \emph{Bott bundle} of $V$).

To simplify notation, we set $E_0:= 2^k\pi_V^*E_\C\to \partial W=M^V$ and $g:=f^V:\partial W\to B\Gamma$. It follows from \eqref{partialeeee}, after possibly taking direct sums of $\mathcal{E}_{C^*_{{\bf full}}(\Gamma)}$ and $\mathcal{E}_{C^*_{{\bf full}}(\Gamma)}'$ with trivial bundles, that there is an isomorphism 
\[\alpha:\mathcal{E}_{C^*_{{\bf full}}(\Gamma)}|_{\partial W}\oplus E_0\otimes g^*\mathcal{L}\rightarrow \mathcal{E}_{C^*_{{\bf full}}(\Gamma)}'|_{\partial W}\oplus E_\C^V\otimes g^*\mathcal{L}.\]
In particular, 
$$\xi:=(\mathcal{E}_{C^*_{{\bf full}}(\Gamma)},\mathcal{E}_{C^*_{{\bf full}}(\Gamma)}',E^V_\C,E_0,\alpha)$$ 
is a cocycle for $K^0(W,\partial W,\mu_\mathcal{L})$. Let $S_{M^V}\to M^V$ denote the spin$^c$-structure of $M^V$. We note that Proposition \ref{tauvversusbott} implies that an explicit cycle representing the image of $\delta(W,\xi,g)$ under $K_1^{geo}(B\Gamma) \xrightarrow{\sim}K_1^h(B\Gamma)$ is given by
\[(M^V,E_\C^V\otimes S_{M^V},g)\dot{\cup} -(M^V,E_0\otimes S_{M^V},g).\]

Define the closed ball bundle $M^B:=\bar{B}(V\oplus 1_\R)$ and let $\pi_B:M^B\to M$ denote the projection. The map $g=f^V:M^V\to B\Gamma$ extends to a map $f^B:=f\circ \pi^B:M^B\to B\Gamma$ and the bundle $E_0\to M^V$ extends to a bundle $E_B:=2^k \pi_B^*E\to M^B$. As such, we obtain an easy cycle 
$$x:=(M^B, (E_B\otimes (f^B)^*\mathcal{L},E_0,\id),g)\quad\mbox{for}\quad\mathcal{S}^{geo}_0(\Gamma, C^*_{\bf full}).$$ 

Choose a decoration 
$$\Xi=\left(\xi,(D_\mathcal{E},D_{\mathcal{E}'},D_{E^V},D_{0}),(A^\mathcal{E},A^{\mathcal{E}'},\tilde{A})\right),$$ 
such that $D^V=D_{E^V}$. Since $\mu(M^V,E_\C^V,g)=\mu(M^V,E_0,g)=0$, we can choose $\tilde{A}=A_0\dot{\cup}-A$, see Theorem \ref{leichpiazzass}. We choose a decoration of $(E_B\otimes (f^B)^*\mathcal{L},E_0,\id)$ of the form
$$\Xi_B=\left((E_B\otimes (f^B)^*\mathcal{L},E_0,\id), (D_B,D_0),(A_B,A_0)\right).$$ 
Define 
\begin{align*}
y:&=-\ind_{APS}(W,\Xi,g)-\ind_{APS}(M^B,\Xi_B,g)\in K_0(C^*_{\bf full}(\Gamma)),\\
&n:=(\sigma_1-\sigma_2)_*\ind_{APS}(W,\Xi,g)\quad\mbox{and} \quad n_B:=(\sigma_1-\sigma_2)_*\ind_{APS}(M^B,\Xi_B,g).
\end{align*} 
By definition, modulo the disjoint union/direct sum relation, we have
\begin{align*}
\tilde{\Phi}_{C^*_{{\bf full}}}&\left((W,\Xi,g)\dot{\cup}r(y)\right)\\
=&(M^V,S^V_{\C\ell},f^V,D^V,A,n)\dot{\cup}-(M^V,E_0\otimes S_{M^V},f^V,D_0,A_0,0)\\
&\qquad\qquad\qquad\qquad\qquad\qquad\qquad\qquad\qquad\qquad\dot{\cup}(\emptyset,\emptyset,\emptyset,\emptyset,\emptyset,-n-n_B)\\
=&(M^V,S^V_{\C\ell},f^V,D^V,A,0)\dot{\cup}-(M^V,E_0\otimes S_{M^V},f^V,D_0,A_0,n_B)
\end{align*}
It is clear that $(M^V,E_0\otimes S_{M^V},f^V,D_0,A_0,n_B)=\Phi_{C^*_{{\bf full}}(\Gamma)}(x)$. The proof is complete upon proving that $x\sim_{bor}0$ in $\mathcal{S}_0^{geo}(\Gamma, C^*_{{\bf full}})$.

We set $M^{B^2}:=\bar{B}(V\oplus 1_\R^2)$ and let $\pi_{B^2}:M^{B^2}\to M$ denote the projection. Introduce the notation $E_{B^2}:=2^k\pi_{B^2}^*E\to M^{B^2}$. We note that $\partial M^{B^2} =M^B\cup_{M^V} M^B$ in a way compatible with the projection maps. We also set $g_2:=f\circ \pi_{B^2}|_{M^B}$. It can be directly verified that 
$$x=\partial ((M^{B^2}, M^B), (E_{B^2}\otimes g_2^*\mathcal{L}, E_B,\alpha_B),g_2)$$
for a suitable $\alpha_B$ constructed from $\alpha$.
\end{proof}

\subsection{Vanishing results}

We summarize the vanishing results for (stable) relative $\eta$-invariants, which can be inferred from results in the previous section; these results have previously been proved using different techniques in \cite{PS}. The results for the stable relative $\eta$-invariants are as follows:

\begin{theorem}
\label{vanishingstablerho}
Assume that $\Gamma$ is a torsion-free group having the full Baum-Connes property. If $(M,S_{\C\ell},g,D)$ is a $K_1^h(B\Gamma)$-cycle with a Dirac operator. Then 
$$\mu_{C^*_{{\bf full}}}(M,S_{\C\ell},g)=0\quad\Rightarrow \quad\rho^s_{\sigma_1,\sigma_2}(M,S_{\C\ell},g,D)=0.$$
\end{theorem}

\begin{proof}
It follows from Proposition \ref{indexsolutionstablerhovanishes} and Theorem \ref{solvingindexproblem} that if $\mu_{C^*_{{\bf full}}}(M,S_{\C\ell},g)=0$ then $\rho^s_{\sigma_1,\sigma_2}(M,S_{\C\ell},g,D)=0$. 
\end{proof}

\begin{remark}
\label{muzerostablerhostable}
Theorem \ref{vanishingstablerho} trivially implies Lemma \ref{stablerhoinvariantlemma} under the stronger assumption that $\mu_{C^*_{\bf full}}$ is an isomorphism because $\rho^s_{\sigma_1,\sigma_2}(M,S_{\C\ell},g,D)=0$ whenever $\mu_{C^*_{\bf full}}(M,S_{\C\ell},f)=0$. 
\end{remark}

The following corollary of Theorem \ref{vanishingstablerho} was proved in \cite{PS} using a suitable choice of smoothing operator; we state it here for completeness.

\begin{cor}
\label{vanishingrhothm}
Assume that $\Gamma$ is a torsion-free group having the full Baum-Connes property. Then
\begin{enumerate}
\item The relative $\eta$-invariant for signature operators (see Notation \ref{signarho}) is homotopy invariant.
\item The relative $\eta$-invariant vanishes on manifolds of positive scalar curvature.
\end{enumerate}
\end{cor}

\begin{remark}
\label{finalremark513}
A connection to work in \cite{BRvonN} should be mentioned. Namely, using the group defined in \cite[Section 4.2]{BRvonN} (and discussed in Remark \ref{ell2group}) and the $L^2$-Atiyah-Patodi-Singer index theorem of Ramachandran \cite{ramachan}, the proof above can be altered to obtain analogous rigidity results for $\ell^2$-relative $\eta$-invariants. These results were first proved in \cite{KesvN}, while a ``Higson and Roe" type proof is presented in \cite{BRvonN}. Of course, such a proof would require additional work and we will not pursue the full details here.
\end{remark}

\paragraph{\textbf{Acknowledgements}}
The authors wish to express their gratitude towards Heath Emerson, Nigel Higson, and Thomas Schick for discussions. They also thank the Courant Centre of G\"ottingen, the Leibniz Universit\"at Hannover and the Graduiertenkolleg 1463 (\emph{Analysis, Geometry and String Theory}) for facilitating this collaboration.

\end{document}